\documentclass[12pt,twoside]{amsart}
\usepackage{amsmath,amsthm,amsmath,amsrefs}

\usepackage{calc}
\usepackage{setspace}

\usepackage{amsbsy,amsmath,amsfonts,amsthm,amssymb,color}

\theoremstyle{plain}
\newtheorem{mythe}{Theorem}[section]
\newtheorem{lem}[mythe]{Lemma}

\newtheorem{cor}[mythe]{Corollary}

\newtheorem{pro}[mythe]{Proposition}

\theoremstyle{definition}

\usepackage{mathrsfs}

\newcommand{\ee}{\varepsilon}

\newcommand{\bN}{\mathbb{N}}
\newcommand{\cT}{\mathcal{T}}

\newcommand{\cF}{\mathcal{F}}

\newcommand{\cS}{\mathcal{S}}
\newcommand{\cB}{\mathcal{B}}
\newcommand{\bT}{\mathbb{T}}
\newcommand{\cK}{\mathcal{K}}

\newcommand{\la}{\langle}
\newcommand{\ra}{\rangle}

\newcommand{\cX}{\mathcal{X}}
\newcommand{\cY}{\mathcal{Y}}
\newcommand{\cH}{\mathcal{H}}
\newcommand{\cA}{\mathcal{A}}

\newcommand{\what}{\widehat}
\newcommand{\bofh}{\cB(\cH)}

\newcommand{\cU}{\mathcal{U}}
\newcommand{\bC}{\mathbb{C}}
\newcommand{\bR}{\mathbb{R}}
\newcommand{\cG}{\mathcal{G}}
\newcommand{\cR}{\mathcal{R}}
\newcommand{\id}{\text{id}}

\newcommand{\bZ}{\mathbb{Z}}

\newcommand{\cJ}{\mathcal{J}}
\newcommand{\cV}{\mathcal{V}}

\newcommand{\tr}{\text{tr}}

\newcommand{\spn}{\text{span }}

\usepackage{tikz}
\usepackage{tikz-cd}

\newdimen\Squaresize \Squaresize=14pt
\newdimen\Thickness \Thickness=0.4pt

\def\Square#1{\hbox{\vrule width \Thickness
   \vbox to \Squaresize{\hrule height \Thickness\vss
      \hbox to \Squaresize{\hss#1\hss}
   \vss\hrule height\Thickness}
\unskip\vrule width \Thickness} \kern-\Thickness}

\def\Vsquare#1{\vbox{\Square{$#1$}}\kern-\Thickness}

\title[Unitary Correlation Sets]{Unitary Correlation Sets}

\author{Samuel J. Harris}
\address{University of Waterloo \\
Department of Pure Mathematics \\
Waterloo, Ontario \\
Canada  N2L 3G1}
\email{sj2harri@uwaterloo.ca}
\author{Vern I. Paulsen}
\address{University of Waterloo \\
Department of Pure Mathematics \\
Waterloo, Ontario \\
Canada  N2L 3G1}
\email{vpaulsen@uwaterloo.ca}
\thanks{Both authors were supported in part by NSERC}



\begin{document}
\begin{abstract}
The unitary correlation sets defined by the first author in conjunction with tensor products of $\cU_{nc}(n)$ are further studied.  We show that Connes' embedding problem is equivalent to deciding whether or not two smaller versions of the unitary correlation sets are equal.  Moreover, we obtain the result that Connes' embedding problem is equivalent to deciding whether or not two cross norms on $M_n \otimes M_n$ are equal for all $n \geq 2$.
\end{abstract}

\maketitle

\section*{Introduction}

Connes' embedding problem \cite{connes} is one of the most important problems in the theory of operator algebras. The problem is the following: does every finite von Neumann algebra $M$ with separable predual embed into the ultrapower of the hyperfinite $II_1$ factor in a way that preserves the trace on $M$? This problem is equivalent to many other open problems in many areas of mathematics. One such area is quantum information theory, as demonstrated by recent results from \cite{fritz}, \cite{junge} and \cite{ozawa}, which show that the embedding problem is intimately related with one of Tsirelson's problems regarding quantum bipartite correlations. In particular, Connes' embedding problem has a positive answer if and only if the set $C_{qc}(n,m)$ of quantum bipartite correlations in the commuting model for $n$ inputs and $m$ outputs can be approximated by the set $C_q(n,m)$ of such correlations in the finite-dimensional tensor product model \cites{fritz,junge,ozawa}.

An analogous theory of unitary correlation sets was developed by the first author in \cite{harris}.  It was shown in \cite{harris} that Connes' embedding problem is equivalent to deciding whether the set $UC_{qc}(n,n)$ of unitary correlations in the commuting model is equal to the closure of the set $UC_q(n,n)$ of unitary correlations in the finite-dimensional tensor product model, for all $n \geq 2$.

Our main result is Theorem \ref{connescrossnorms}, which states that Connes' embedding problem is equivalent to deciding whether a certain compression $B_{qc}(n,n)$ of $UC_{qc}(n,n)$ is equal to the closure of the analogous compression $B_q(n,n)$ of $UC_q(n,n)$.  Moreover, certain cross norms on $M_n \otimes M_n$ arise from the sets $\overline{B_q(n,n)}$ and $B_{qc}(n,n)$, and it is shown that the embedding problem is equivalent to determining whether or not these cross norms are equal on $M_n \otimes M_n$, for all $n \geq 2$.  Drawing on techniques in \cite{CLP}, we show that $B_q(n,m) \neq B_{qc}(n,m)$ for all $n,m \geq 2$ and that $B_q(n,m)$ is not closed.  This result is one way in which the unitary correlation sets differ greatly from the quantum bipartite correlation sets.

In this paper, we draw on some results from operator system tensor theory and quotient theory, as well as recent work in quantum information theory regarding the embezzlement of entanglement of states.  In Section \S1, we review some results regarding operator system tensor products, quotients and coproducts that we will use.  In Section \S2, we give a brief introduction to the probabilistic correlation sets arising in Tsirelson's problems.  Section \S3 gives some properties of the smaller unitary correlation sets $B_q(n,m)$ and $B_{qc}(n,m)$, along with other related unitary correlation sets.  Moreover, the correspondence between these correlation sets and cross norms on $M_n \otimes M_m$ is given.  We relate Connes' embedding problem to determining whether or not $\overline{B_q(n,n)}=B_{qc}(n,n)$ in Section \S4.  Finally, in Section \S5, we use the theory of embezzling entanglement of states from \cite{CLP} and \cite{vandam} to demonstrate several separations between the various unitary correlation sets.

\section{Preliminaries}
The theory of operator systems has many connections to Connes' embedding problem.  In this section, we will give a brief introduction to certain aspects of the theory; namely, we will introduce duality, tensor products, quotients and coproducts.  First, suppose that $\cS$ is an operator system, and let $\cS^d$ be its Banach space dual.  The space $\cS^d$ can always be endowed with the structure of a matrix ordered space \cite[Lemma 4.2, Lemma 4.3]{choieffros}.  The involution on $\cS^d$ is given by $f^*(s)=\overline{f(s^*)}$ for each $f \in \cS^d$ and $s \in \cS$.  We say that an element $(f_{ij}) \in M_n(\cS^d)$ is positive provided that the map $F:\cS \to M_n$ given by $F(s)=(f_{ij}(s))$ is completely positive.  With these notions, $\cS^d$ is a matrix ordered space.  Moreover, if $\dim(\cS)<\infty$, then $\cS^d$ is an operator system with order unit given by a faithful state on $\cS$ \cite{choieffros}.  In this case, the canonical map $i:\cS \to \cS^{dd}$ is a complete order isomorphism.

Throughout this paper, we will be considering three tensor products: the minimal, the commuting, and the maximal tensor products.  First, we briefly summarize some of the theory of tensor products of operator systems from \cite{KPTT} and \cite{quotients} that we shall need.

Let $\mathcal{O}$ denote the category of operator systems with unital, completely positive maps as the morphisms.  An \textbf{operator system tensor product} is a map $\tau:\mathcal{O} \times \mathcal{O} \to \mathcal{O}$, $(\cS,\cT) \mapsto \cS \otimes_{\tau} \cT$, satisfying the following conditions:
\begin{enumerate}
\item
$\cS \otimes_{\tau} \cT$ has the structure of an operator system on the vector space $\cS \otimes \cT$, with adjoint given by $(s \otimes t)^*=s^* \otimes t^*$ and Archimedean matrix order unit given by $1_{\cS} \otimes 1_{\cT}$;
\item
If $X \in M_p(\cS)_+$ and $Y \in M_q(\cT)_+$, then $X \otimes Y:=(X_{ij} \otimes Y_{k\ell})_{(i,j),(k,\ell)}$ is in $M_{pq}(\cS \otimes_{\tau} \cT)_+$; and
\item
If $\varphi:\cS \to M_p$ and $\psi:\cT \to M_q$ are ucp maps, then $\varphi \otimes \psi:\cS \otimes_{\tau} \cT \to M_{pq}$ is ucp.
\end{enumerate}
An operator system tensor product $\tau$ is said to be \textbf{symmetric} if, for every $(\cS,\cT) \in \mathcal{O}$, the canonical map $s \otimes t \mapsto t \otimes s$ extends to a complete order isomorphism from $\cS \otimes_{\tau} \cT$ onto $\cT \otimes_{\tau} \cS$.  An operator system tensor product $\tau$ is said to be \textbf{functorial} if it satisfies the following property:
\begin{itemize}
\item
If $\cS_1,\cS_2,\cT_1$ and $\cT_2$ are operator systems and $\varphi:\cS_1 \to \cT_1$ and $\psi:\cS_2 \to \cT_2$ are ucp maps, then $\varphi \otimes \psi:\cS_1 \otimes_{\tau} \cS_2 \to \cT_1 \otimes_{\tau} \cT_2$ is ucp.
\end{itemize}

Following \cite{KPTT}, we now define the minimal, commuting and maximal tensor products of operator systems.  Suppose that $\cS$ and $\cT$ are operator systems, and let $\iota:\cS \to \bofh$ and $\kappa:\cT \to \cB(\cK)$ be complete order embeddings, where $\cH$ and $\cK$ are Hilbert spaces.  The \textbf{minimal tensor product} of $\cS$ and $\cT$, denoted by $\cS \otimes_{\min} \cT$, is the operator system arising from the inclusion $(\iota \otimes \kappa)(\cS \otimes \cT) \subseteq \cB(\cH \otimes \cK)$.  Equivalently, an element $X \in M_n(\cS \otimes_{\min} \cT)_{sa}$ is positive if and only if for every pair of ucp maps $\varphi:\cS \to M_p$ and $\psi:\cT \to M_q$, we have $(\varphi \otimes \psi)^{(n)}(X) \in M_{npq}^+$.  In particular, the operator system $\cS \otimes_{\min} \cT$ is independent of the choice of Hilbert spaces $\cH$ and $\cK$, and independent of the complete order embeddings $\iota$ and $\kappa$ \cite[Theorem 4.4]{KPTT}.

Given two linear maps $\varphi:\cS \to \bofh$ and $\psi:\cT \to \bofh$, we will let $\varphi \cdot \psi:\cS \otimes \cT \to \bofh$ be the product map defined on simple tensors by $(\varphi \cdot \psi)(s \otimes t)=\varphi(s)\psi(t)$.  With this notion in hand, we define the \textbf{commuting tensor product} of $\cS$ and $\cT$ to be the operator system $\cS \otimes_c \cT$ such that $X \in M_n(\cS \otimes_c \cT)_{sa}$ is positive if and only if $(\varphi \cdot \psi)^{(n)}(X) \in M_n(\bofh)_+$ for every pair of ucp maps $\varphi:\cS \to \bofh$ and $\psi:\cT \to \bofh$ with commuting ranges.

Finally, the \textbf{maximal tensor product} of $\cS$ and $\cT$ is defined as the operator system $\cS \otimes_{\max} \cT$ such that $X \in M_n(\cS \otimes_{\max} \cT)_{sa}$ is positive if and only if for every $\ee>0$, there are $S_{\ee} \in M_p(\cS)_+$, $T_{\ee} \in M_q(\cT)_+$ and a linear map $A_{\ee}:\bC^n \to \bC^p \otimes \bC^q$ such that $$X+\ee I_n=A_{\ee}^*(S_{\ee} \otimes T_{\ee})A_{\ee}.$$

Finally, if $\alpha$ and $\beta$ are operator system tensor products, then we write $\alpha \leq \beta$ if for every pair of operator systems $\cS$ and $\cT$, the identity map $\text{id}:\cS \otimes_{\beta} \cT \to \cS \otimes_{\alpha} \cT$ is completely positive.  Each of $\min,c,\max$ are symmetric functorial operator system tensor products \cite{KPTT}.  It is also shown in \cite{KPTT} that $$\min \leq c \leq \max.$$

Before introducing coproducts of operator systems, it is helpful to consider the theory of operator system quotients.  Suppose that $\varphi:\cS \to \cT$ is a surjective ucp map between operator systems, and let $\cJ=\ker(\varphi)$.  We can endow the quotient vector space $\cS/\cJ$ with an operator system structure.  We define an involution on $\cS/\cJ$ as $(\dot{x})^*=\dot{(x^*)}$ for each $\dot{x} \in \cS/\cJ$.  We let $$D_n(\cS,\cJ)=\{ \dot{X} \in M_n(\cS/\cJ)_{sa}: X+K \in M_n(\cS)_+ \text{ for some } K \in M_n(\cS)_{sa}\}.$$
Finally, we define the set of positive elements of $M_n(\cS/\cJ)$ to be the set $$C_n(\cS,\cJ)=\{ \dot{X} \in M_n(\cS/\cJ)_{sa}: \dot{X}+\ee \dot{I_n} \in D_n(\cS,\cJ), \, \forall \ee>0\}.$$
Then by \cite{quotients}, $\cS/\cJ$ is an operator system with order unit $\dot{1}$.

Given an operator system $\cS$, a surjective ucp map $\varphi:\cS \to \cT$ and a kernel $\cJ=\ker(\varphi)$ as above, we will say that $\cJ$ is \textbf{completely order proximinal} provided that $D_n(\cS,\cJ)=C_n(\cS,\cJ)$ for all $n \in \bN$.  While the notion of a first isomorphism theorem fails in general for operator systems, the following weaker version still holds.

\begin{pro}
\emph{(Kavruk-Paulsen-Todorov-Tomforde, \cite{quotients})}
\label{firstiso}
If $\varphi:\cS \to \cT$ is a ucp map and $\cJ=\ker(\varphi)$, then the induced map $\dot{\varphi}:\cS/\cJ \to \cT$ given by $\dot{\varphi}(\dot{x})=\varphi(x)$ is ucp.
\end{pro}

In general, given a surjective ucp map $\varphi:\cS \to \cT$, we will say that $\varphi$ is a \textbf{complete quotient map} if $\dot{\varphi}:\cS/\ker(\varphi) \to \cT$ defined as above is a complete order isomorphism.  We have the following relation between complete quotient maps and complete order injections:

\begin{mythe}
\emph{(Farenick-Paulsen, \cite{FP})}
\label{quotientadjoint}
Let $\varphi:\cS \to \cT$ be a surjective ucp map between finite-dimensional operator systems.  Then $\varphi$ is a complete quotient map if and only if the adjoint mapping $\varphi^d:\cT^d \to \cS^d$ given by $[\varphi^d(f)](s)=f(\varphi(s))$ for all $f \in \cT^d$ and $s \in \cS$ is a complete order embedding.
\end{mythe}

We will also consider coproducts of operator systems.  These are akin to free products of $C^{\ast}$-algebras amalgamated over the unit.  For simplicity, we will consider the coproduct of finitely many operator systems $\cS_1,...,\cS_n$ with $n \geq 2$.  More information can be found in \cite{FKPTwep}.  Let $\cS_1,...,\cS_n$ be operator systems, and let $e_i$ denote the order unit of $\cS_i$.  The \textbf{coproduct} of $\cS_1,...,\cS_n$ is an operator system $\bigoplus_1 \{S_i\}_{i=1}^n$, together with unital complete order embeddings $\kappa_i:\cS_i \to \bigoplus_1 \{S_i\}_{i=1}^n$, which satisfies the following universal property: for any operator system $\cR$ and any collection of ucp maps $\varphi_i:\cS_i \to \cR$ for $1 \leq i \leq n$, there is a unique ucp map $\varphi:\bigoplus_1 \{S_i\}_{i=1}^n \to \cR$ such that $\varphi(\kappa_i(s_i))=\varphi_i(s_i)$ for all $s_i \in \cS_i$ and $1 \leq i \leq n$.

The coproduct $\bigoplus_1 \{S_i\}_{i=1}^n$ can always be realized as a complete quotient of the direct sum $\bigoplus_{i=1}^n \cS_i$ of the operator systems $\cS_1,...,\cS_n$.  Let $\cS=\bigoplus_{i=1}^n \cS_i$ be the direct sum of $\cS_1,...,\cS_n$.  This operator system has order unit $(e_1,...,e_n)$.  We let $$\cJ=\left\{ (x_i)_{i=1}^n \in \cS: x_i=\lambda_i e_i, \, \lambda_i \in \bC, \, \sum_{i=1}^n \lambda_i=0\right\}.$$

Note that an element in $\cS$ is positive if and only if each coordinate is positive.  Therefore, $\cJ$ has no positive elements except $0$.  It follows by \cite[Proposition 2.4]{kavruk} that $\cJ$ is a completely order proximinal kernel of a ucp map.  Hence, we may form the quotient operator system $\cS/\cJ$.

\begin{mythe}
\emph{(Farenick-Kavruk-Paulsen-Todorov, \cite{FKPTwep})}
\label{coproduct}
Let $\cS_1,...,\cS_n$ be operator systems, let $\cS=\bigoplus_{i=1}^n \cS_i$ and let $\cJ$ be defined as above.  Then $\bigoplus_1 \{\cS_i\}_{i=1}^n$ is completely order isomorphic to $\cS/\cJ$.
\end{mythe}

\section{Probabilistic Correlation Sets}

Before examining the unitary correlation sets from \cite{harris} in detail, it is helpful to consider the correlation sets arising in Tsirelson's problems as a comparison.  We give a brief introduction to these quantum bipartite correlation sets below; see \cites{fritz, junge, tsirelson93} for more information.

We recall that a \textbf{projection valued measure} with $m$ outputs is a collection of projections $\{P_i\}_{i=1}^m$ on a Hilbert space $\cH$ such that $\sum_{i=1}^m P_i=I_{\cH}$.  The set of \textbf{quantum correlations} in $n$ inputs and $m$ outputs, denoted by $C_q(n,m)$, is defined as the set of all coordinates of the form $$(\la E_{a,x} \otimes F_{b,y} \xi,\xi \ra)_{a,x,b,y},$$
where $\{E_{a,x}\}_{a=1}^m$ is a PVM on a finite-dimensional Hilbert space $\cH_A$ for each $1 \leq x \leq n$, $\{F_{b,y}\}_{b=1}^m$ is a PVM on a finite-dimensional Hilbert space $\cH_B$ for each $1 \leq y \leq n$, and $\xi \in \cH_A \otimes \cH_B$ is a unit vector.  The set of \textbf{quantum spatial correlations} $C_{qs}(n,m)$ is defined in the same manner, except that we no longer assume that $\cH_A$ and $\cH_B$ are finite-dimensional.

The set of \textbf{quantum commuting correlations} in $n$ inputs and $m$ inputs, denoted by $C_{qc}(n,m)$, is the set of all coordinates of the form $$(\la E_{a,x}F_{b,y}\xi,\xi \ra)_{a,b,x,y},$$ where for each $1 \leq x,y \leq n$, the collections $\{E_{a,x}\}_{a=1}^m$ and $\{F_{b,y}\}_{b=1}^m$ are PVM's on the same Hilbert space $\cH$, $\xi \in \cH$ is a unit vector, and $E_{a,x}$ commutes with $F_{b,y}$ for all $a,b,x,y$.  For convenience, we will also let $C_{qa}(n,m)=\overline{C_q(n,m)}$.

\begin{mythe}
\emph{(Ozawa, \cite{ozawa})}
The following are equivalent.
\begin{enumerate}
\item
Connes' embedding problem has a positive answer.
\item
$C_{qa}(n,m)=C_{qc}(n,m)$ for all $n,m \geq 2$.
\item
$C_{qa}(n,m)=C_{qc}(n,m)$ for a fixed $m \geq 2$ and all $n \geq 2$.
\end{enumerate}
\end{mythe}

There is a natural link between the sets $C_t(n,m)$ for $t \in \{qa,qc\}$ and operator system tensor products.  Consider $\cA=\ast_{i=1}^n \ell_m^{\infty}$, the free product of $n$ copies of $\ell_m^{\infty}$, amalgamated over the identity.  Let $e_i^{(k)}$ be the generator of the $i$-th coordinate of the $k$-th copy of $\ell_m^{\infty}$ in $\cA$.  Let $$\cF_{n,m}=\spn \{ e_i^{(k)}: 1 \leq i \leq m, \, 1 \leq k \leq n\}.$$

By results from \cite{fritz} and \cite{junge}, the above correlation sets correspond to states on various operator system structures on $\cF_{n,m} \otimes \cF_{n,m}$.  To simplify the notation, we let $e_{a,x}$ correspond to the generator of the $a$-th coordinate in the $x$-th copy of $\ell_m^{\infty}$ on the left of the tensor product.  We let $f_{b,y}$ correspond to the generator of the $b$-th coordinate in the $y$-th copy of $\ell_m^{\infty}$ on the right of the tensor product.  Then $$C_{qa}(n,m)=\{ (s(e_{a,x} \otimes f_{b,y}))_{a,b,x,y}: s \in \cS(\cF_{n,m} \otimes_{\min} \cF_{n,m})\},$$
whereas $$C_{qc}(n,m)=\{(s(e_{a,x} \otimes f_{b,y}))_{a,b,x,y}: s \in \cS(\cF_{n,m} \otimes_c \cF_{n,m}) \}.$$

We define the set of \textbf{quantum maximal correlations} with $n$ inputs and $m$ outputs to be the set $$C_{qmax}(n,m)=\{(s(e_{a,x} \otimes f_{b,y}))_{a,b,x,y}: s \in \cS(\cF_{n,m} \otimes_{\max} \cF_{n,m})\}.$$

We will show that $C_{qmax}(n,m)$ is precisely the set of all non-signalling box correlation probabilities in the sense of \cite{tsirelson80}.  Before we can prove this result, we need a description of the dual of $\cF_{n,m}$.  To this end, the following is very useful.

\begin{mythe}
\emph{(Farenick-Kavruk-Paulsen-Todorov, \cite{FKPTwep})}
$\cF_{n,m}$ is completely order isomorphic to the coproduct of $n$ copies of $\ell_m^{\infty}$.  In other words, $$\cF_{n,m} \simeq \bigoplus_1 \{ \ell_m^{\infty}\}_{i=1}^n.$$
\end{mythe}

\begin{cor}
The dual of $\cF_{n,m}$ is completely order isomorphic to $$\cS_{n,m}=\{ (\gamma_1,...,\gamma_n) \in \bigoplus_{i=1}^n \ell_m^{\infty}: \sum_{j=1}^n \gamma_i(j)=\sum_{j=1}^n \gamma_k(j), \, \forall 1 \leq i,k \leq m\}.$$
\end{cor}

\begin{proof}
Since $\cF_{n,m}$ is a complete quotient of $\bigoplus_{i=1}^n \ell_m^{\infty}$ by the kernel $\cJ$ in the sense of Theorem \ref{coproduct}, the adjoint map gives a complete order embedding of $\cF_{n,m}^d$ into $\left( \bigoplus_{i=1}^n \ell_m^{\infty} \right)^d \simeq \bigoplus_{i=1}^n \ell_m^{\infty}$ by Theorem \ref{quotientadjoint}.  Thus, the vector space dual of $\cF_{n,m}$ with the operator system structure inherited from $\bigoplus_{i=1}^n \ell_m^{\infty}$ is the operator system dual of $\cF_{n,m}$.  This space is none other than the annihilator of $\cJ$, which is exactly $\cS_{n,m}$.
\end{proof}

For $n,m \in \bN$, we define the set of \textbf{non-signalling box probabilities} to be the set of coordinates $\{(p(a,b|x,y)): 1 \leq a,b \leq m, \, 1 \leq x,y \leq n\}$, subject to the following conditions:
\begin{itemize}
\item
$p(a,b|x,y) \geq 0$, for all $a,b,x,y$;
\item
$\sum_{a,b=1}^m p(a,b|x,y)=1$ for all $x,y$;
\item
$\sum_{a=1}^m p(a,b|x,y)=\sum_{a=1}^n p(a,b|x',y)$ for all $b,x,x',y$; and
\item
$\sum_{b=1}^m p(a,b|x,y)=\sum_{b=1}^m p(a,b|x,y')$ for all $a,x,y,y'$.
\end{itemize}
We denote by $C_{nsb}(n,m)$ the set of all non-signalling box probabilities.

\begin{mythe}
For all $n,m \geq 2$, $C_{qmax}(n,m)=C_{nsb}(n,m)$.
\end{mythe}

\begin{proof}
Let $s \in \cS(\cF_{n,m} \otimes_{\max} \cF_{n,m})$; we will show that $(s(e_{a,x} \otimes f_{b,y}))_{a,b,x,y}$ is in $C_{nsb}(n,m)$.  As $e_{a,x}$ and $f_{b,y}$ are positive elements in $\cF_{n,m}$, we have $e_{a,x} \otimes f_{b,y} \in (\cF_{n,m} \otimes_{\max} \cF_{n,m})_+$.  Since $1=\sum_{a=1}^m e_{a,x}$ and $1=\sum_{b=1}^m f_{b,y}$, it is easy to see that $1 \otimes 1=\sum_{a,b=1}^m e_{a,x} \otimes f_{b,y}$, so that the first two conditions of $C_{nsb}(n,m)$ hold for $(s(e_{a,x} \otimes f_{b,y}))_{a,b,x,y}$.  We show that the third condition also holds; the fourth condition is similar.  Let $x \neq x'$; then since $\cF_{n,m}$ is a coproduct, $\sum_{a=1}^m e_{a,x}=\sum_{a=1}^m e_{a,x'}=1$.  Hence, $$\sum_{a=1}^m e_{a,x} \otimes f_{b,y}=\sum_{a=1}^m e_{a,x'} \otimes f_{b,y}.$$
The third and fourth conditions follow, so that $(s(e_{a,x} \otimes f_{b,y}))_{a,b,x,y}$ is an element of $C_{nsb}(n,m)$.

Conversely, suppose that $(p(a,b|x,y))_{a,b,x,y}$ is in $C_{nsb}(n,m)$.  Define a function $s:\cF_{n,m} \otimes_{\max} \cF_{n,m} \to \bC$ by $s(e_{a,x} \otimes f_{b,y})=p(a,b|x,y)$.  The third and fourth conditions guarantee that $s$ is a functional on $\cF_{n,m} \otimes \cF_{n,m}$. We see that $$s(1)=\sum_{a,b=1}^m s(e_{a,x} \otimes f_{b,y})=\sum_{a,b=1}^m p(a,b|x,y)=1.$$
Hence, $s$ is unital.  Identify $s$ with its image in $\left( \bigoplus_{j=1}^n \ell_m^{\infty} \right) \otimes_{\min} \left( \bigoplus_{j=1}^n \ell_m^{\infty} \right)$.  This element is positive if and only if every coordinate is non-negative.  The coordinates of $s$ in $\left( \bigoplus_{j=1}^n \ell_m^{\infty} \right) \otimes_{\min} \left( \bigoplus_{j=1}^n \ell_m^{\infty} \right)$ are precisely the elements $s(e_{a,x} \otimes f_{b,y})=p(a,b|x,y)$, so that $s$ is positive.  Hence, $s$ is a state on $\cF_{n,m} \otimes_{\max} \cF_{n,m}$, which shows that $(p(a,b|x,y))_{a,b,x,y} \in C_{qmax}(n,m)$.  Therefore, $C_{qmax}(n,m)=C_{nsb}(n,m)$.
\end{proof}


\section{Unitary Correlation Norms and Connes' Embedding Problem}

For $n \in \bN$ with $n \geq 2$, we let $\cU_{nc}(n)$ denote the universal $C^{\ast}$-algebra with generators $\{u_{ij}:1 \leq i,j \leq n \}$ such that the matrix $U=(u_{ij})$ is unitary in $M_n(\cU_{nc}(n))$.  This $C^{\ast}$-algebra was defined by L. Brown in \cite{brown}.  It possesses the following universal property: if $\cA$ is a unital $C^{\ast}$-algebra and $\{a_{ij}: 1 \leq i,j \leq n \}$ is a subset of $\cA$ such that $(a_{ij})$ is unitary in $M_n(\cA)$, then there is a unique unital $*$-homomorphism $\pi:\cU_{nc}(n) \to \cA$ such that $\pi(u_{ij})=a_{ij}$.  We let $\cV_n=\spn \{1,u_{ij},u_{ij}^*\}_{i,j=1}^n \subseteq \cU_{nc}(n)$ be the operator system spanned by the generators of $\cU_{nc}(n)$.  The operator system $\cV_n$ has the following universal property:

\begin{pro}
\emph{(Harris, \cite{harris})}
Let $\{a_{ij}:1 \leq i,j \leq n \}$ be a subset of a unital $C^{\ast}$-algebra $\cA$ such that $\|(a_{ij})\| \leq 1$.  Then there is a unique ucp map $\psi:\cV_n \to \cA$ such that $\psi(u_{ij})=a_{ij}$ for all $1 \leq i,j \leq n$.
\end{pro}

The following theorem shows how $\cV_n$ can be obtained as a quotient of $M_{2n}$.

\begin{mythe}
\emph{(Harris, \cite{harris})}
\label{vnquotient}
The map $\varphi_n:M_{2n} \to \cV_n$ given by $$\varphi_n(E_{ij})=\begin{cases} \frac{1}{2n} 1 & i=j \\ \frac{1}{2n}u_{i(j-n)} & i \leq n, \, j \geq n+1 \\ \frac{1}{2n}u_{j(i-n)}^* & i \geq n+1, \, j \leq n \\ 0 & \text{otherwise} \end{cases}$$
is a complete quotient map.
\end{mythe}

As an immediate corollary, we obtain the following:

\begin{cor}
\label{vnvmquotient}
If $n,m \in \bN$ with $n,m \geq 2$, then $\varphi_n \otimes \varphi_m:M_{2n} \otimes M_{2m} \to \cV_n \otimes_{\max} \cV_m$ is a complete quotient map.
\end{cor}

\begin{proof}
It is straightforward to check that if $\cJ_{2n}$ is the kernel of $\varphi_n$, then $$\ker(\varphi_n \otimes \varphi_m)=\cJ_{2n} \otimes M_{2m}+M_{2n} \otimes \cJ_{2m}.$$
Now, let $X \in M_p(\cV_n \otimes_{\max} \cV_m)$ be strictly positive; that is, assume that $X \geq \ee 1$ for some $\ee>0$.  Then there are $S \in M_k(\cV_n)_+$, $T \in M_q(\cV_m)_+$, and a rectangular matrix $A \in M_{p,kq}$ such that $$X=A(S \otimes T)A^*.$$
Since $\varphi_n$ and $\varphi_m$ are complete quotient maps and $\cJ_{2n}$ and $\cJ_{2m}$ are completely order proximinal, we may find matrices $P,Q$ with entries in $M_{2n}$ and $M_{2m}$, respectively, with quotient images equal to $S$ and $T$ respectively.  Then $X$ is the image of a positive element in $M_{2n} \otimes M_{2m}$, and we are done.
\end{proof}

Recall that if $\cX$ and $\cY$ are Banach spaces, then a \textbf{reasonable cross-norm} on the vector space tensor product $\cX \otimes \cY$ is a norm $\alpha$ on $\cX \otimes \cY$ satisfying the following:
\begin{itemize}
\item
$\alpha(x \otimes y) \leq \|x\| \|y\|$ whenever $x \in \cX$ and $y \in \cY$; and
\item
if $\varphi \in \cX^*$ and $\psi \in \cY^*$, then $\varphi \otimes \psi$ is bounded on $\cX \otimes \cY$, with $$\|\varphi \otimes \psi\|:=\sup \{ |(\varphi \otimes \psi)(u)|: u \in \cX \otimes \cY, \, \alpha(u) \leq 1\} \leq \| \varphi \| \|\psi \|.$$
\end{itemize}

Given a reasonable cross-norm $\alpha$ on $\cX \otimes \cY$, we denote by $\cX \otimes_{\alpha} \cY$ the vector space $\cX \otimes \cY$ with the norm $\alpha$, and we denote by $\cX \widehat{\otimes}_{\alpha} \cY$ the completion of $\cX \otimes \cY$ with respect to the norm $\alpha$.  Two examples are in order.  The first is the \textbf{projective Banach space tensor norm} $\| \cdot \|_{\pi}$, given by $$\|u\|_{\pi}=\inf \left\{ \sum_{i=1}^n \|x_i\|\|y_i\|: x_i \in \cX, \, y_i \in \cY, \, n \in \bN, \, u=\sum_{i=1}^n x_i \otimes y_i \right\}.$$
The second example is the \textbf{injective Banach space tensor norm} $\| \cdot \|_{\ee}$, given by $$\|u\|_{\ee}=\sup \{ |(\varphi \otimes \psi)(u)|: \varphi \in \cX^*, \, \psi \in \cY^*, \, \| \varphi \| \leq 1, \, \|\psi \| \leq 1\}.$$  It is well-known (see, for example, \cite[Proposition 6.1]{ryan}) that a norm $\alpha$ on $\cX \otimes \cY$ is a reasonable cross-norm if and only if $\| \cdot \|_{\ee} \leq \alpha(\cdot) \leq \| \cdot \|_{\pi}$.  We will also say that a reasonable cross-norm $\alpha$ that is defined for all pairs of Banach spaces is \textbf{functorial} if, for all Banach spaces $\cX_1,\cX_2,\cY_1,\cY_2$ and bounded linear maps $S:\cX_1 \to \cX_2$ and $T:\cY_1 \to \cY_2$, the map $S \otimes T:\cX_1 \otimes \cY_1 \to \cX_2 \otimes \cY_2$ extends to a bounded linear map from $\cX_1 \what{\otimes}_{\alpha} \cX_2$ to $\cY_1 \what{\otimes}_{\alpha} \cY_2$ with norm $\|S\| \|T\|$.  Both the injective and projective Banach space tensor norms are functorial \cite[p. 129]{ryan}.

We will explore properties of the unitary correlation sets defined in \cite{harris}.  By way of notation, whenever $\cH$ is a Hilbert space and $U=(U_{ij}) \in M_n(\bofh)$ is unitary, we will let $\mathfrak{B}_n(U)=\{I_{\cH}\} \cup \{U_{ij},U_{ij}^*\}_{i,j=1}^n$.  We define $$UC_q(n,m)=\{ (\la (X \otimes Y)\psi,\psi \ra)_{X \in \mathfrak{B}_n(U), \, Y \in \mathfrak{B}_m(V)}\},$$
where $U \in M_n(\cB(\cH_A))$ and $V \in M_m(\cB(\cH_B))$ are unitary, $\cH_A$ and $\cH_B$ are finite-dimensional Hilbert spaces, and $\psi \in \cH_A \otimes \cH_B$ is a unit vector.  We define $UC_{qs}(n,m)$ to be the set of correlations in $UC_q(n,m)$, only dropping the requirement that $\cH_A$ and $\cH_B$ be finite-dimensional.  For the commuting model, we define $$UC_{qc}(n,m)=\{ (\la XY\psi,\psi \ra)_{X \in \mathfrak{B}_n(U), \, Y \in \mathfrak{B}_m(V)}\},$$
where $U \in M_n(\bofh)$ and $V \in M_m(\bofh)$ are unitary, $\cH$ is a Hilbert space, $\psi \in \cH$ is a unit vector, and $XY=YX$ for all $X \in \mathfrak{B}_n(U)$ and $Y \in \mathfrak{B}_m(V)$.  We can also define a local model.  For local correlations, we let $UC_{loc}(n,m)$ be the set of correlations in $UC_{qc}(n,m)$ such that $C^*(\mathfrak{B}_n(U) \cup \mathfrak{B}_m(V))$ is a commutative $C^{\ast}$-algebra.

For each of the above correlation sets $UC_t(n,m)$, we will consider the smaller set $B_t(n,m)$ obtained by only considering $X \in \{U_{ij}: 1 \leq i,j \leq n\}$ and $Y \in \{V_{k\ell}:1 \leq k,\ell \leq m\}$.

To define quantum maximal unitary correlation sets, we will require a slightly different approach.  We let $$\cG_{n,m}=\{ x \otimes y: x \in \{1\} \cup \{u_{ij},u_{ij}^*\}_{i,j=1}^n, \, y \in \{1\} \cup \{v_{k\ell},v_{k\ell}^*\}_{k,\ell=1}^m \}.$$ 
We let $UC_{qmax}(n,m)$ be the set of all coordinates of the form $$(s(x))_{x \in \cG_{n,m}},$$
where $s$ is a state on $\cV_n \otimes_{\max} \cV_m$.  Similarly, we let $$B_{qmax}(n,m)=\{(s(u_{ij} \otimes v_{k\ell}))_{(i,j),(k,\ell)}: s \in \cS(\cV_n \otimes_{\max} \cV_m) \}.$$

We define $UC_{qmin}(n,m)$ to be the set of all coordinates of the form $(s(x))_{x \in \cG_{n,m}}$,
where $s$ is a state on $\cV_n \otimes_{\min} \cV_m$, and $$B_{qmin}(n,m) = \{ (s(u_{ij} \otimes v_{k\ell}))_{(i,j),(k,\ell)}: s \in \cS(\cV_n \otimes_{\min} \cV_m) \}.$$
Some of the known properties of these sets are summarized in the following theorem.  Aside from the presence of $UC_{qmax}(n,m)$, the proof of this theorem can be found in \cite{harris}.

\begin{mythe}
\emph{(Harris, \cite{harris})}
Let $n,m \geq 2$.  Then $\overline{UC_{q}(n,m)} = \overline{UC_{qs}(n,m)}= UC_{qmin}(n,m)$ and
$$UC_q(n,m) \subseteq UC_{qs}(n,m) \subseteq UC_{qmin}(n,m) \subseteq UC_{qc}(n,m) \subseteq UC_{qmax}(n,m).$$
Moreover, each of these sets is convex, and $UC_{qc}(n,m)$ is closed.
\end{mythe}

\begin{proof} The last containment is the only result not shown in \cite{harris}.
To show that $UC_{qc}(n,m) \subseteq UC_{qmax}(n,m)$, we use the fact that $UC_{qc}(n,m)$ corresponds to states on $\cV_n \otimes_c \cV_m$ \cite{harris}, while $UC_{qmax}(n,m)$ corresponds to states on $\cV_n \otimes_{\max} \cV_m$.  As every state on $\cV_n \otimes_c \cV_m$ is a state on $\cV_n \otimes_{\max} \cV_m$, we obtain the desired inclusion.  Since $UC_{qmax}(n,m)$ corresponds to a state space, it is clearly convex, as required.
\end{proof}

To be consistent with the notation used for probabilistic correlation sets, we set $UC_{qa}(n,m)=UC_{qmin}(n,m)$ and $B_{qa}(n,m) = B_{qmin}(n,m)$.  The link between unitary correlation sets and Connes' embedding problem can be summarized as follows:

\begin{mythe}
\emph{(Harris, \cite{harris})}
The following are equivalent.
\begin{enumerate}
\item
Connes' embedding problem has a positive answer.
\item
$UC_{qa}(n,m)=UC_{qc}(n,m)$ for all $n,m \geq 2$.
\item
$UC_{qa}(n,n)=UC_{qc}(n,n)$ for all $n \geq 2$.
\end{enumerate}
\end{mythe}

A simple but crucial observation is that for $t_1,t_2 \in \{loc,qa,qc,qmax\}$, if $UC_{t_1}(n,m)=UC_{t_2}(n,m)$, then $B_{t_1}(n,m)=B_{t_2}(n,m)$.  Hence, one way to separate $UC_{t_1}(n,m)$ and $UC_{t_2}(n,m)$ is by separating the sets $B_{t_1}(n,m)$ and $B_{t_2}(n,m)$.  We will see that, for Connes' embedding problem, it suffices to consider the sets $B_t(n,m)$.  Moreover, the sets $B_t(n,m)$ for $t \in \{loc,qa,qc,qmax\}$ have a very special structure, as seen below.

\begin{mythe}
\label{correlationnorms}
For $t \in \{loc,qa,qc,qmax\}$, the set $B_t(n,m)$ is the unit ball of a norm $\| \cdot \|_t$ on $M_n \otimes M_m$.  Moreover, $\| \cdot \|_{loc}$ is the norm arising from the projective Banach space tensor product $M_n \otimes_{\pi} M_m$.
\end{mythe}

\begin{proof}
As each set $B_t(n,m)$ corresponds to images of states, it is easy to see that $B_t(n,m)$ is convex.  Since $0$ is a contraction in $M_n$, there is a state $\eta_n:\cV_n \to \bC$ with $\eta(u_{ij})=0$ for all $i,j$.  By functoriality of the min tensor product, $\eta_n \otimes \eta_m:\cV_n \otimes_{\min} \cV_m \to \bC \otimes \bC=\bC$ is a state, which corresponds to the matrix $0 \in M_{nm}$.  Each entry of a matrix in $B_t(n,m)$ must have modulus at most $1$, so the set $B_t(n,m)$ is clearly compact in $M_{nm}$.  It remains to show that $0$ is an interior point in $B_t(n,m)$.  Since $B_{loc}(n,m)$ is the smallest of the correlation sets, it suffices to prove that $0$ is an interior point for $B_{loc}(n,m)$.  Since $\bC$ is a commutative $C^{\ast}$-algebra, any pair of unitary matrices $X \in M_n$ and $Y \in M_m$ satisfies $X \otimes Y \in B_{loc}(n,m)$.  Using the fact that the convex hull of the unitaries in $M_n$ is the unit ball of the operator norm in $M_n$, we see that $\{ X \otimes Y \in M_n \otimes M_m: \|X\|,\|Y\| \leq 1 \} \subseteq B_{loc}(n,m)$.  It is well-known that the closed convex hull of the former set is the unit ball of the projective Banach space tensor product norm \cite[Proposition 2.2]{ryan}; hence, it follows that $0$ is an interior point for $B_{loc}(n,m)$.  Therefore, each $B_t(n,m)$ is the unit ball of a norm $\| \cdot \|_t$ on $M_{nm}$.

It remains to show that $\| \cdot \|_{loc}=\| \cdot \|_{\pi}$.  To this end, let $\cA$ be a unital, commutative $C^{\ast}$-algebra, and let $U \in M_n(\cA)$ and $V \in M_m(\cA)$ be unitary.  Note that $\cA \simeq C(X)$ for some compact Hausdorff space $X$, so that the extreme points of $\cS(\cA)$ are just the evalation functionals $\{\delta_x: x \in X \}$.  The matrix in $B_{loc}(n,m)$ arising from one of these states corresponding to $U$ and $V$ is $(\delta_x(u_{ij}v_{k\ell}))=(\delta_x(u_{ij})\delta_x(v_{k\ell}))$.  Note that $(\delta_x(u_{ij}))$ and $(\delta_v(v_{k\ell}))$ are contractions in $M_n$ and $M_m$ respectively, so that $(\delta_x(u_{ij})\delta_x(v_{k\ell}))$ is of the form $A \otimes B$ where $A \in M_n$ and $B \in M_m$ are contractions.  Taking the closed convex hull of the pure states on $C(X)$, we see that every element of $B_{loc}(n,m)$ is in the closed convex hull of $\{A \otimes B: A \in M_n, \, B \in M_m, \, \|A\| \leq 1, \, \|B\| \leq 1\}$.  This shows that $\| \cdot \|_{loc}$ is the projective Banach space tensor norm on $M_n \otimes M_m$, as desired.
\end{proof}

We will see later that if $t \neq qmax$, then $\| \cdot\|_t$ cannot be unitarily invariant.  However, all of these norms satisfy a weaker condition.

\begin{pro}
\label{locallyunitarilyinvariant}
For $t \in \{loc,qa,qc\}$, the norm $\| \cdot \|_t$ is locally unitarily invariant on $M_n \otimes M_m$; i.e., for any unitaries $U_1,U_2 \in M_n$, unitaries $V_1,V_2 \in M_m$ and $X \in M_n \otimes M_m$, we have $$\|(U_1 \otimes V_1)X(U_2 \otimes V_2)\|_t=\|X\|_t.$$
\end{pro}

\begin{proof}
First, let $s$ be a state on $\cV_n \otimes_c \cV_m$.  Then there is a Hilbert space $\cH$, unitaries $U=(U_{ij}) \in \cB(\bC^n \otimes \cH)$ and $V=(V_{k\ell}) \in \cB(\cH \otimes \bC^m)$, and a unit vector $\psi \in \cH$ such that $s(u_{ij} \otimes v_{k\ell})=\la U_{ij}V_{k\ell}\psi,\psi \ra$.  Let $X=(s(u_{ij} \otimes v_{k\ell}))_{(i,j),(k,\ell)} \in M_n \otimes M_m$.  We will show that $X[(\alpha_{ij}) \otimes I] \in UC_{qc}(n,m)$ whenever $(\alpha_{ij})$ is a unitary matrix in $M_n$; the rest of the cases will follow.  Define $\what{U}_{ij}=\sum_{p=1}^n U_{ip} \alpha_{p j}$.  Then $(\what{U}_{ij})=U(\alpha_{ij})$ is unitary, and $\what{U}_{ij}V_{k\ell}=V_{k\ell} \what{U}_{ij}$.  It follows that $$X((\alpha_{ij}) \otimes I)=(\la U_{ij}V_{k\ell} \psi,\psi \ra)(\alpha_{ij} \otimes I)=(\la \what{U}_{ij}V_{k\ell}\psi,\psi \ra) \in UC_{qc}(n,m).$$
If the entries of $U$ and $V$ generate a commutative $C^{\ast}$-algebra, then the same is true for the entries of $\what{U}=(\what{U}_{ij})$ and $V$, so that $B_{loc}(n,m)$ is locally unitarily invariant.  If we assume that $X \in B_{qs}(n,m)$, then $s(u_{ij} \otimes v_{k\ell})$ can be written as $\la (U_{ij} \otimes V_{k\ell})\psi,\psi \ra$, where $U=(U_{ij}) \in M_n(\cB(\cH_A))$ and $V=(V_{k\ell}) \in M_m(\cB(\cH_B))$ are unitary, $\cH_A$ and $\cH_B$ are Hilbert spaces, and $\psi \in \cH_A \otimes \cH_B$ is a unit vector.  Applying the same approach as above, the matrix $X((\alpha_{ij}) \otimes I)$ arises from a state induced by a tensor product of representations, so that $X((\alpha_{ij}) \otimes I) \in B_{qs}(n,m)$.  Therefore, the set $B_{qs}(n,m)$ is also locally unitarily invariant.  The fact that $B_{qa}(n,m)$ is locally unitarily invariant follows by taking limits and using the fact that matrix multiplication is continuous in any norm topology on $M_{nm}$.
\end{proof}

Like the norm $\| \cdot \|_{loc}$, each of the norms $\| \cdot \|_t$ must be a reasonable cross-norm.

\begin{mythe}
\label{crossnorm}
For $t \in \{loc,qa,qc,qmax\}$, $\| \cdot \|_t$ is a reasonable cross-norm on $M_n \otimes M_m$.  Moreover, if $\| \cdot \|$ denotes the operator norm on $M_{nm}$, then $\| \cdot \| \leq \| \cdot \|_t$.
\end{mythe}

\begin{proof}
Since $B_{loc}(n,m) \subseteq B_t(n,m)$ for $t \in \{qa,qc,qmax\}$, we know that $\|X \otimes Y\|_t \leq 1$ whenever $X \in M_n$ and $Y \in M_m$ satisfy $\|X\| \leq 1$ and $\|Y\| \leq 1$.  Hence, $\| \cdot \|_t$ is a cross-norm.  Once we show that $\| \cdot \| \leq \| \cdot \|_t$, we will have $\| \cdot \|_{\ee} \leq \| \cdot \|_t \leq \| \cdot \|_{\pi}$, where $\| \cdot \|_{\ee}$ is the injective Banach space tensor norm, which shows that $\| \cdot \|_t$ is a reasonable cross-norm.  To see that $\| \cdot \|_t \geq \| \cdot \|$, we need only show that $\| \cdot \| \leq \| \cdot \|_{qmax}$, since $\| \cdot \|_{qmax}$ defines the smallest $\| \cdot \|_t$.  Let $X \in B_{qmax}(n,m)$; then there is a state $s \in \cS(\cV_n \otimes_{\max} \cV_m)$ with $X=(s(u_{ij} \otimes v_{k\ell}))_{(i,j),(k,\ell)}$.  Any operator system tensor product is an operator space tensor product \cite[Proposition 3.4]{KPTT}.  Since $\|(u_{ij})\|=1$ and $\|(v_{k\ell})\|=1$, we must have $\|(u_{ij} \otimes v_{k\ell})\|=1$ in $M_{nm}(\cV_n \otimes_{\max} \cV_m)$.  Since $s$ is completely contractive, we see that $\|X\| \leq 1$ in $M_{nm}$, and the result follows.
\end{proof}

The lower bound in Theorem \ref{crossnorm} is attained by the norm arising from $B_{qmax}(n,m)$.

\begin{mythe}
\label{qmax}
For $n,m \geq 2$, the norm $\| \cdot \|_{qmax}$ with unit ball equal to $B_{qmax}(n,m)$ is the operator norm on $M_{nm}$.  In other words, $$B_{qmax}(n,m)=\{ X \in M_{nm}: \|X\| \leq 1 \}.$$
\end{mythe}

\begin{proof}
Theorem \ref{crossnorm} shows that $B_{qmax}(n,m) \subseteq \{X \in M_{nm}: \|X\| \leq 1\}$.  For the reverse inclusion, let $X \in M_n \otimes M_m$ with operator norm at most $1$.  We may write $$X=\sum_{\substack{1 \leq i,j \leq n \\ 1 \leq k,\ell \leq m}} x_{ijk\ell} E_{ij} \otimes E_{k\ell}.$$
Define the element $$\chi=\sum_{i,j,k,\ell} x_{ijk\ell} \otimes E_{12} \otimes E_{ij} \otimes E_{12} \otimes E_{k\ell} \in M_2 \otimes M_n \otimes M_2 \otimes M_m,$$
and let $$P=I_2 \otimes I_n \otimes I_2 \otimes I_m+\chi+\chi^*.$$
Since $\| \chi \| \leq 1$ in $M_{2n} \otimes M_{2m}$, $P$ is positive in $M_{2n} \otimes M_{2m}$.  Therefore, the corresponding map $\gamma_P:M_2 \otimes M_n \otimes M_2 \otimes M_m \to \bC$ with Choi matrix equal to $P$ is a positive linear functional; moreover, $\gamma_P(I_2 \otimes I_n \otimes I_2 \otimes I_m)=4mn$.

Let $\cJ_{2n}=\ker \varphi_n$, where $\varphi_n:M_{2n} \to \cV_n$ is the complete quotient map in Theorem \ref{vnquotient}.  We claim that $\gamma_P(\cJ_{2n} \otimes M_{2m}+M_{2n} \otimes \cJ_{2m})=0$, so that $\gamma_P$ induces a positive linear functional $\widetilde{\gamma}_P$ on $\cV_n \otimes_{\max} \cV_m$.  To show this, we will show that $\gamma_P$ annihilates $\cJ_{2n} \otimes M_{2m}$; the other part is similar.  We may write $\cJ_{2n} \otimes M_{2m}$ as the set of all elements of the form $$\{(E_{11} \otimes A+E_{22} \otimes B) \otimes W\},$$
where $A,B \in M_n$, $W \in M_{2m}$ and $\tr(A)+\tr(B)=0\}$.
Applying $\gamma_P$ to an element $C$ of $\cJ_{2n} \otimes M_{2m}$, we obtain $$\gamma_P(C)=\tr(A)\tr(W)+\tr(B)\tr(W)=(\tr(A)+\tr(B))\tr(W)=0.$$
It follows that $\gamma_P(\ker(\varphi_n \otimes \varphi_m))=0$.  By Proposition \ref{firstiso} and Corollary \ref{vnvmquotient}, the induced functional $\widetilde{\gamma}_P:\cV_n \otimes_{\max} \cV_m \to \bC$ is positive with $\widetilde{\gamma}_P(1)=\gamma_P(I_2 \otimes I_n \otimes I_2 \otimes I_m)=4mn$.  Let $s=\frac{1}{4mn} \widetilde{\gamma}_P$, which is a state on $\cV_n \otimes_{\max} \cV_m$.  We observe that $$s(u_{ij} \otimes v_{k\ell})=\frac{1}{4mn} \widetilde{\gamma}_P(u_{ij} \otimes v_{k\ell})=\widetilde{\gamma}_P \left( \frac{1}{2n} u_{ij} \otimes \frac{1}{2m} v_{k\ell} \right).$$
Recall that the quotient image of $E_{12} \otimes E_{ij} \in M_2 \otimes M_n$ under the map $\varphi_n$ is $\frac{1}{2n} u_{ij}$, and similarly, the quotient image of $E_{12} \otimes E_{k\ell} \in M_2 \otimes M_m$ under the map $\varphi_m$ is $\frac{1}{2m} v_{k\ell}$.  Therefore, $$s(u_{ij} \otimes v_{k\ell})=\gamma_P(E_{12} \otimes E_{ij} \otimes E_{12} \otimes E_{k\ell})=x_{ijk\ell},$$
so that $(s(u_{ij} \otimes v_{k\ell}))=X$.  We conclude that $X \in B_{qmax}(n,m)$, as desired.
\end{proof}

A careful examination of Theorem \ref{vnquotient} and the proof of Theorem \ref{qmax} shows that for $X \in B_{qmax}(n,m)$, there is a state on $\cV_n \otimes_{\max} \cV_m$ such that $(s(u_{ij} \otimes v_{k\ell}))_{(i,j),(k,\ell)}=X$ and $s(x)=0$ for every $x \in \cG_{n,m} \setminus \{u_{ij} \otimes v_{k\ell} \}_{i,j,k,\ell} $.  The following proposition shows that such a state can always be found for elements of $B_t(n,m)$, where $t \in \{loc,qa,qc\}$.

\begin{pro}
Let $t \in \{ loc, qa,qc\}$ and $X \in B_t(n,m)$.  Then there is a state $s$ on $\cV_n \otimes_c \cV_m$ such that $s(u_{ij} \otimes 1)=0=s(1 \otimes v_{k\ell})$ and $s(u_{ij} \otimes v_{k\ell}^*)=0$ for all $i,j,k,\ell$, and $(s(u_{ij} \otimes v_{k\ell}))=X$.  If $X \in B_{qa}(n,m)$, then $s$ can be taken to be a state on $\cV_n \otimes_{\min} \cV_m$.
\end{pro}

\begin{proof}
Using the containments $B_{loc}(n,m) \subseteq B_{qc}(n,m)$ and $B_{qa}(n,m) \subseteq B_{qc}(n,m)$, there is a state $\omega$ on $\cV_n \otimes_c \cV_m$ with $(\omega(u_{ij} \otimes v_{k\ell}))=X$.  Let $U_{ij}$ and $V_{k\ell}$ be operators on a Hilbert space $\cH$ and let $\psi \in \cH$ be a unit vector such that $U=(U_{ij})$ and $V=(V_{k\ell})$ are unitaries in $M_n(\bofh)$ and $M_m(\bofh)$ respectively; $U_{ij}V_{k\ell}=V_{k\ell}U_{ij}$ for all $i,j,k,\ell$; and $\omega(u_{ij} \otimes v_{k\ell})=\la U_{ij}V_{k\ell}\psi,\psi \ra$.  For $\theta \in [0,2\pi]$, define $\omega_{\theta}$ to be the state on $\cV_n \otimes_c \cV_m$ corresponding to the unitaries $U_{\theta}=(e^{i \theta}U_{ij})$ and $V_{\theta}=(e^{-i \theta}V_{k\ell})$ and unit vector $\psi$.  Then the entries of $U_{\theta}$ and $V_{\theta}$ still $*$-commute, and $\la (U_{\theta})_{ij}(V_{\theta})_{k\ell}\psi,\psi \ra=X_{(i,j),(k,\ell)}$ for all $i,j,k,\ell$.  It is immediate that $\omega_{\theta}(u_{ij} \otimes 1)=e^{i \theta} \omega(u_{ij} \otimes 1)$, $\omega_{\theta}(1 \otimes v_{k\ell})=e^{-i \theta} \omega(1 \otimes v_{k\ell})$ and $\omega_{\theta}(u_{ij} \otimes v_{k\ell}^*)=e^{2i \theta} \omega(u_{ij} \otimes v_{k\ell}^*)$.  Define $s:\cV_n \otimes \cV_m \to \bC$ by $$s(x)=\frac{1}{2\pi}\int_0^{2\pi} \omega_{\theta}(x) \, d\theta,$$ which defines a state on $\cV_n \otimes_c \cV_m$, satisfying $s(1 \otimes v_{k\ell})=0=s(u_{ij} \otimes 1)$ and $s(u_{ij} \otimes v_{k\ell}^*)=0$.  If $X \in B_{qa}(n,m)$, then the state $s$ can be taken to be be a limit of states on $\cV_n \otimes_{\min} \cV_m$, so that $s$ is a state on $\cV_n \otimes_{\min} \cV_m$.
\end{proof}


\section{Connes' Embedding Problem}

We now move towards another equivalent statement of Connes' embedding problem.  We will show that the equality of the $qa$ and $qc$ norms on $M_n \otimes M_n$ is equivalent to a positive answer to the embedding problem.  First, we adopt some notation.  Let $F_{\infty}$ denote the free group on a countably infinite number of generators, and let $(w_i)_{i=1}^{\infty}$ be a set of universal generators for $F_{\infty}$.  We define the following operator systems: \begin{align*}
\cX_n&=\spn(\{1\} \cup \{w_i \otimes w_j,w_i^* \otimes w_j^*\}_{i,j=1}^n) \subseteq C^*(F_n) \otimes_{\min} C^*(F_n), \\ \cY_n&=\spn(\{1\} \cup \{w_i \otimes w_j,w_i^* \otimes w_j^*\}_{i,j=1}^n) \subseteq C^*(F_n) \otimes_{\max} C^*(F_n), \\
\cX_{\infty}&=\spn(\{1\} \cup \{w_i \otimes w_j, w_i^* \otimes w_j^*\}_{i,j=1}^{\infty}) \subseteq C^*(F_{\infty}) \otimes_{\min} C^*(F_{\infty}) \text{ and} \\ \cY_{\infty}&=\spn(\{1\} \cup \{w_i \otimes w_j, w_i^* \otimes w_j^*\}_{i,j=1}^{\infty}) \subseteq C^*(F_{\infty}) \otimes_{\max} C^*(F_{\infty}).
\end{align*}

\begin{pro}
\label{correlationorderiso}
Let $n \geq 2$.  If $B_{qa}(n,n)=B_{qc}(n,n)$, then the formal identity map $\id:\cX_n \to \cY_n$ is an order isomorphism.
\end{pro}

\begin{proof}
Since $B_{qa}(n,n)=B_{qc}(n,n)$, the formal identity map $$\id:\spn (\{1\} \cup \{u_{ij} \otimes v_{k\ell},u_{ij}^* \otimes v_{k\ell}^*\}_{i,j,k,\ell}) \to \cU_{nc}(n) \otimes_{\max} \cU_{nc}(n)$$ from $\cU_{nc}(n) \otimes_{\min} \cU_{nc}(n)$ is an order isomorphism onto its range.  By the proof of \cite[Theorem 4.10]{harris}, there are ucp maps $\psi_n:C^*(F_n) \to \cU_{nc}(n)$ and $\pi_n:\cU_{nc}(n) \to C^*(F_n)$ such that $\id_{C^*(F_n)}=\pi_n \circ \psi_n$.  Moreover, $\psi_n(w_i)=u_{ii}$ and $\pi_n(u_{ij})=\delta_{ij} w_i$.  By functoriality of the min and max tensor products, $\psi_n \otimes \psi_n:C^*(F_n) \otimes_{\min} C^*(F_n) \to \cU_{nc}(n) \otimes_{\min} \cU_{nc}(n)$ and $\psi_n \otimes \psi_n:C^*(F_n) \otimes_{\max} C^*(F_n) \to \cU_{nc}(n) \otimes_{\max} \cU_{nc}(n)$ are complete order embeddings.  Therefore, $\cX_n$ is completely order isomorphic to $\spn (\{1\} \cup \{u_{ii} \otimes v_{jj},u_{ii}^* \otimes v_{jj}^*\}_{i,j=1}^n)$ inside of $\cU_{nc}(n) \otimes_{\min} \cU_{nc}(n)$; the analogous result holds for $\cY_n$ inside of $\cU_{nc}(n) \otimes_{\max} \cU_{nc}(n)$.  It follows that $\id:\cX_n \to \cY_n$ is an order isomorphism.
\end{proof}

We require a few results from \cite{ozawa}.  We first recall that, given a $C^{\ast}$-algebra $\cA$, the \textbf{opposite algebra} of $\cA$, denoted by $\cA^{op}$, is a $C^{\ast}$-algebra with the same $*$-vector space structure as $\cA$, but with multiplication given by $(a^{op}b^{op})=(ba)^{op}$.  A special case of the opposite algebra is for $C^*(F_{\infty})$.  The mapping $w_i \mapsto w_i^*$ extends to a unital $*$-isomorphism from $C^*(F_{\infty})$ onto $C^*(F_{\infty})^{op}$.

\begin{mythe}
\emph{(Ozawa, \cite{ozawa})}
\label{tracialstateonop}
Let $\cA$ be a unital $C^{\ast}$-algebra, and let $\tau$ be a tracial state on $\cA$.  Then the map $s_{\tau}:\cA \otimes_{\max} \cA^{op} \to \bC$ given by $s_{\tau}(a \otimes b^{op})=\tau(ab)$ extends to a state on $\cA \otimes_{\max} \cA^{op}$.
\end{mythe}

We obtain the following description of traces in terms of certain states on $\cY_{\infty}$.

\begin{mythe}
\emph{(Ozawa, \cite{ozawa})}
\label{tracesandstatesonmax}
Let $\cA$ be a separable $C^{\ast}$-algebra, and let $(u_i)_{i=1}^{\infty}$ be a generating sequence of unitaries in the unitary group of $\cA$.  Suppose that $\tau$ is a tracial state on $\cA$.  Then the mapping $w_i \otimes w_j \mapsto \tau(u_iu_j^*)$ extends to a state on $\cY_{\infty}$.
\end{mythe}

\begin{proof}
We let $\sigma:C^*(F_{\infty}) \to \cA$ be the surjective unital $*$-homomorphism given by $\sigma(w_i)=u_i$.  By Theorem \ref{tracialstateonop}, $\tau$ induces a state $s_{\tau}$ on $\cA \otimes_{\max} \cA^{op}$ given by $a \otimes b^{op} \mapsto \tau(ab)$.  Let $\sigma^{op}:(C^*(F_{\infty}))^{op} \to \cA^{op}$ denote the opposite representation of $\sigma$, given by $\sigma^{op}(w_i^{op})=u_i^{op}$.  Then $\sigma \otimes \sigma^{op}:C^*(F_{\infty}) \otimes_{\max} C^*(F_{\infty}) \to \cA \otimes_{\max} \cA^{op}$ is a $*$-homomorphism.  Using the fact that $C^*(F_{\infty})^{op} \simeq C^*(F_{\infty})$, we see that the mapping $w_i \otimes w_j \mapsto \tau(u_iu_j^*)=s_{\tau} \circ (\sigma \otimes \sigma^{op})(w_i \otimes w_j^{op})$ extends to a state on $\cY_{\infty}$.
\end{proof}

The key result that links $\cX_{\infty}$ to Connes' embedding problem is the following.

\begin{mythe}
\emph{(Ozawa, \cite{ozawa})}
\label{satisfyingconnes}
Let $\cA$ be a separable $C^{\ast}$-algebra with a countable dense sequence $(u_i)_{i=1}^{\infty}$ of unitaries and a tracial state $\tau$.  Then $(\pi_{\tau}(\cA)'',\tau)$ satisfies Connes' embedding problem if and only if the mapping $w_i \otimes w_j \mapsto \tau(u_iu_j^*)$ extends to a state on $\cX_{\infty}$.
\end{mythe}

In order to use Theorem \ref{satisfyingconnes}, we must ensure that each $\cX_n$ and $\cY_n$ can be considered inside of the respective tensor product of $C^*(F_{\infty})$.

\begin{lem}
For each $n \geq 2$, the formal identity maps $\id:\cX_n \to \cX_{\infty}$ and $\id:\cY_n \to \cY_{\infty}$ are complete order embeddings.
\end{lem}

\begin{proof}
Since the minimal operator system tensor product is injective and $\cX_n \subseteq C^*(F_n) \otimes_{\min} C^*(F_n) \subseteq C^*(F_{\infty}) \otimes_{\min} C^*(F_{\infty})$, the result immediately follows for $\cX_n$.  Now, the canonical embedding $F_n \hookrightarrow F_{\infty}$ and canonical quotient map $F_{\infty} \to F_n$ give rise to $*$-homomorphisms $\pi_n:C^*(F_n) \to C^*(F_{\infty})$ and $\sigma_n:C^*(F_{\infty}) \to C^*(F_n)$ with $\sigma_n \circ \pi_n=\id_{C^*(F_n)}$.  By functoriality of the maximal tensor product, $\pi_n \otimes \pi_n$ and $\sigma_n \otimes \sigma_n$ are ucp with respect to the maximal tensor product.  Therefore, the following diagram commutes:

\begin{center}
$\begin{tikzcd}[column sep=small]
& C^*(F_{\infty}) \otimes_{\max} C^*(F_{\infty}) \drar{\sigma_n \otimes \sigma_n} & \\
C^*(F_n) \otimes_{\max} C^*(F_n) \urar{\pi_n \otimes \pi_n}  \ar{rr}{\id} & & C^*(F_n) \otimes_{\max} C^*(F_n)
\end{tikzcd}$
\end{center}

Hence, $C^*(F_n) \otimes_{\max} C^*(F_n)$ is completely order isomorphic to the image of $C^*(F_n) \otimes C^*(F_n)$ in $C^*(F_{\infty}) \otimes_{\max} C^*(F_{\infty})$.  Restricting to $\cY_n$ shows that the formal identity map $\id:\cY_n \to \cY_{\infty}$ is a complete order isomorphism onto its range.
\end{proof}

We are now ready for the main result of this section.

\begin{mythe}
\label{connescrossnorms}
The following are equivalent.
\begin{enumerate}
\item
Connes' embedding problem has a positive answer.
\item
$UC_{qa}(n,m)=UC_{qc}(n,m)$ for all $n,m \geq 2$.
\item
$B_{qa}(n,m)=B_{qc}(n,m)$ for all $n,m \geq 2$.
\item
$B_{qa}(n,n)=B_{qc}(n,n)$ for all $n \geq 2$.
\item
$M_n \otimes_{qa} M_n=M_n \otimes_{qc} M_n$ isometrically for all $n \geq 2$.
\end{enumerate}
\end{mythe}

\begin{proof}
The equivalence of (1) and (2) is by \cite[Theorem 6.12]{harris}.  Clearly (2) implies (3) and (3) implies (4).  Two norms on the same space are equal if and only if their closed unit balls are identical, so (4) is equivalent to (5).  Hence, it remains to show that (4) implies (1).

Suppose that $B_{qa}(n,n)=B_{qc}(n,n)$ for all $n \geq 2$.  By Proposition \ref{correlationorderiso}, the formal identity map $\id:\cX_n \to \cY_n$ is an order isomorphism for all $n \geq 2$.  Let $\cA$ be a separable $C^{\ast}$-algebra with a countable dense sequence $(u_i)_{i=1}^{\infty}$ of unitaries, and let $\tau$ be a tracial state on $\cA$.  By Theorem \ref{tracesandstatesonmax}, the mapping $w_i \otimes w_j \mapsto \tau(u_iu_j^*)$ extends to a state on $\cY_{\infty}$.

Consider the $C^{\ast}$-algebra $\cA_n=C^*(u_1,...,u_n)$, which has a generating sequence of unitaries given by $(v_i)_{i=1}^{\infty}$, where $v_i=u_i$ for $i \leq n$ and $v_i=1$ for $i>n$.  Define $s_n:\cY_{\infty} \to \bC$ to be the unital, self-adjoint mapping given by $w_i \otimes w_j \mapsto \tau(v_iv_j^*)$.  Then $s_n$ is a state by Theorem \ref{tracesandstatesonmax}.  Restricting to $\cY_n$, $(s_n)_{|\cY_n}$ must be a state on $\cX_n$.  By the Hahn-Banach theorem, we may extend $(s_n)_{|\cY_n}$ to a state on $\cX_{\infty}$, which we will denote by $\omega_n$.  If $x \in \cX_{\infty}$, then $x=\lambda 1+\sum_{i,j=1}^n (\lambda_{ij} w_i \otimes w_j+\mu_{ij} w_i^* \otimes w_j^*)$ for some $n$, so that $x \in \cX_n$.  It follows that $\lim_{m \to \infty} \omega_m(x)=\omega_n(x)$.  Hence, $(\omega_m)_{m=1}^{\infty}$ converges pointwise to the linear map $\omega:\cX_{\infty} \to \bC$ given by $\omega(1)=1$, $\omega(w_i \otimes w_j)=\tau(u_iu_j^*)$ and $\omega(w_i^* \otimes w_j^*)=\overline{\tau(u_iu_j^*)}$.  Since the state space of $\cX_{\infty}$ is $w^*$-closed, $\omega$ is a state.  Therefore, the mapping $w_i \otimes w_j \mapsto \tau(u_iu_j^*)$ extends to a state on $\cX_{\infty}$, so that $(\pi_{\tau}(\cA)'',\tau)$ satisfies Connes' embedding problem.  Since $\cA$ was an arbitrary $C^{\ast}$-algebra with separable unitary group, we see that Connes' embedding problem must have a positive answer.  Hence, (4) implies (1).
\end{proof}


\section{Separating the Unitary Correlation Sets}

In this section, we will use results from \cite{CLP} to show that $B_{qs}(n,m) \neq B_{qc}(n,m)$ for all $n,m \geq 2$; moreover, we will show that $B_{qs}(n,m)$ is not closed.  Attempts to obtain comparable results for the probabilistic quantum correlation sets given in Tsirelson's problem have a long history and are less definitive.  It was only recently shown by W. Slofstra \cite{slofstra} that there are $n,m \in \bN$ such that $C_{qs}(n,m) \neq C_{qc}(n,m)$, but for which pairs these sets are not equal is unknown.  
Since our paper was posted (arXiv:1612.02791), W. Slofstra has posted a new paper (arXiv:1703.08618) showing that there exist $n_1,n_2,k_1,k_2$ for which the set $C_{qs}(n_1,n_2,k_1,k_2)$ is not closed, where $n_1$ is the number of inputs for Alice, $n_2$ is the number of inputs for Bob, $k_1$ is the number of outputs for Alice, and $k_2$ is the number of outputs for Bob \cite{slofstra17}. (Slofstra's counterexample has $n_1=184$, $n_2=235$, $k_1=8$ and $k_2=2$.) The two analogous problems for unitary correlation sets have a negative answer for every $n,m \geq 2$, as we will see below.

Before we establish separations between some of the various unitary correlation sets, we require some terminology involving state embezzlement, as described in \cite{CLP}.  We give a somewhat simplified embezzlement framework here.  Suppose that Alice and Bob each have access to a finite-dimensional Hilbert space; we will always assume that Alice's space is $\bC^n$ and Bob's space is $\bC^m$ for some $n,m \geq 2$.  Suppose that Alice and Bob have access to a resource Hilbert space $\cR$, and are able to act on the system $\bC^n \otimes \cR \otimes \bC^m$ locally.  We consider whether there is a unit vector $\psi \in \cR$ such that Alice and Bob's operations can send $e_1 \otimes \psi \otimes e_1$ to $\sum_{i,j} \alpha_{ij} e_i \otimes \psi \otimes e_j$, where $\sum_{i,j} |\alpha_{ij}|^2=1$.  We will say that there is a \textbf{perfect embezzlement protocol in a finite-dimensional tensor product model} for $\sum_{i,j} \alpha_{ij} e_i \otimes e_j$ if there is a resource Hilbert space $\cR=\cR_A \otimes \cR_B$, operators $U_{ij} \in \cB(\cR_A)$ and $V_{k\ell} \in \cB(\cR_B)$ for $1 \leq i,j \leq n$ and $1 \leq k,\ell \leq m$ such that $U=(U_{ij})$ and $V=(V_{k\ell})$ are unitary on $\bC^n \otimes \cR_A$ and $\cR_B \otimes \bC^m$ respectively, with
 $$(U \otimes V)(e_1 \otimes \psi \otimes e_1)=\sum_{i,j} \alpha_{ij} e_i \otimes \psi \otimes e_j.$$
We will say that there is a \textbf{perfect embezzlement protocol in a tensor product model} for $\sum_{i,j} \alpha_{ij} e_i \otimes e_j$ if the same conditions are met as above, except that we drop the requirement that $\dim(\cR_A),\dim(\cR_B)<\infty$.  A \textbf{perfect embezzlement protocol in the commuting model} for $\sum_{i,j} \alpha_{ij} e_i \otimes e_j$ will have the same properties as above, except that we drop the assumption that $\cR$ decomposes as a tensor product, and instead assume that $U_{ij},V_{k\ell} \in \cB(\cR)$ for all $i,j,k,\ell$, and that $$(U \otimes I_m)(I_n \otimes V)=(I_n \otimes V)(U \otimes I_m).$$

The next two results relate perfect embezzlement and states on tensor products of $\cV_n$.

\begin{pro}
\emph{(Cleve-Liu-Paulsen, \cite{CLP})}
\label{embezzlecauchyschwarz}
Let $U_{ij},V_{k\ell} \in \cB(\cR)$ for $1 \leq i,j \leq n$ and $1 \leq k,\ell \leq m$ be such that $U=(U_{ij})$ and $V=(V_{k\ell})$ are unitary.  Then $(U \otimes I_m)(I_n \otimes V)=(I_n \otimes V)(U \otimes I_m)$ if and only if $U_{ij}V_{k\ell}=V_{k\ell}U_{ij}$ and $U_{ij}^* V_{k\ell}=V_{k\ell}U_{ij}^*$ for all $i,j,k,\ell$.
\end{pro}

The following is a slight extension of a result from \cite{CLP}.

\begin{pro}
\emph{(Cleve-Liu-Paulsen, \cite{CLP})}
A perfect embezzlement protocol in the commuting model exists for $\sum_{i,j} \alpha_{ij} e_i \otimes e_j$ if and only if there is a state $s \in \cS(\cU_{nc}(n) \otimes_{\max} \cU_{nc}(m))$ such that $s(u_{i1} \otimes v_{j1})=\alpha_{ij}$.
\end{pro}

\begin{proof}
Suppose that $U_{ij},V_{k\ell} \in \cB(\cR)$ are such that $U=(U_{ij})$ and $V=(V_{k\ell})$ are unitary and $\psi \in \cR$ is a unit vector such that $(U \otimes I_m)(I_n \otimes V)=(I_n \otimes V)(U \otimes I_m)$ and $(U \otimes I_m)(I_n \otimes V)(e_1 \otimes \psi \otimes e_1)=\sum_{i,j} \alpha_{ij} e_i \otimes \psi \otimes e_j$.  Then there is a unital $*$-homomorphism $\pi:\cU_{nc}(n) \otimes_{\max} \cU_{nc}(m) \to \cB(\cR)$ such that $\pi(u_{ij} \otimes v_{k\ell})=U_{ij}V_{k\ell}$.  Define the state $s:\cU_{nc}(n) \otimes_{\max} \cU_{nc}(m) \to \bC$ such that $u_{ij} \otimes v_{k\ell} \mapsto \la U_{ij}V_{k\ell}\psi,\psi \ra$.  Since $\cV_n \otimes_c \cV_m$ is completely order isomorphic to its inclusion in $\cU_{nc}(n) \otimes_{\max} \cU_{nc}(m)$ \cite{harris}, we obtain a state $s:\cV_n \otimes_c \cV_m \to \bC$ such that $$s(u_{i1} \otimes v_{j1})=\la U_{i1}V_{j1}\psi,\psi \ra=\alpha_{ij}.$$
Conversely, suppose that such a state $s$ exists.  Then $(s(u_{ij} \otimes v_{k\ell})) \in B_{qc}(n,m)$, so there are unitaries $U=(U_{ij})$ and $V=(V_{k\ell})$ with $U_{ij},V_{k\ell} \in \cB(\cR)$ and $U_{ij}V_{k\ell}=V_{k\ell}U_{ij}$, and a unit vector $\psi \in \cR$ such that, for each $i,j,k,\ell$, we have $s(u_{ij} \otimes v_{k\ell})=\la U_{ij}V_{k\ell}\psi,\psi \ra$. Now, 
\begin{align*}
1=\sum_{i,j} |\alpha_{ij}|^2&=\sum_{i,j} |\la U_{i1}V_{j1}\psi,\psi \ra|^2 \\
&\leq \sum_{i,j} \|U_{i1}V_{j1}\psi \|^2=\left\| (U \otimes I_m)(I_n \otimes V) \begin{pmatrix} \psi \\ 0 \\ \vdots \\ 0 \end{pmatrix} \right\|^2=1, \\
\end{align*}
using the fact that $(U \otimes I_m)(I_n \otimes V)$ is unitary.  Therefore, $|\la U_{i1}V_{j1} \psi,\psi \ra|=\| U_{i1} V_{j1} \psi \|$ for all $i,j$.  Since $\la U_{i1}V_{j1} \psi,\psi \ra=\alpha_{ij}$, by the Cauchy-Schwarz inequality, we must have $U_{i1}V_{j1}\psi=\alpha_{ij}\psi$.  Therefore, we observe that $$(U \otimes I_m)(I_n \otimes V)(e_1 \otimes \psi \otimes e_1)=\sum_{i,j} \alpha_{ij} e_i \otimes \psi \otimes e_j,$$
so a perfect embezzlement protocol exists in the commuting model for the unit vector $\sum_{i,j} \alpha_{ij} e_i \otimes e_j$.
\end{proof}

We now give a proof that in the commuting model, any norm one vector in $\bC^n \otimes \bC^m$ can be perfectly embezzled.  In particular, we give an alternate proof that any norm one vector in $\bC^n \otimes \bC^m$ can be approximately embezzled; i.e., one can use unitaries $(U_{ij})$ and $(V_{k\ell})$ to obtain the mapping $e_1 \otimes \psi \otimes e_1 \mapsto \sum_{i,j} \alpha_{ij} e_i \otimes \psi_{\ee} \otimes e_j$, where $|\la \psi,\psi_{\ee}\ra| \geq 1-\ee$ for a small $\ee>0$.  This fact was first proved in \cite{vandam}, and was reproved in \cite{CLP} for the vector $\frac{1}{\sqrt{n}} \sum_{i=1}^n e_i \otimes e_i$.  Our method of proof here draws on a simplification due to Richard Cleve; we kindly thank him for sharing this simplification.

\begin{mythe}
\label{embezzlement}
Let $n,m \geq 2$ and let $\sum_{i=1}^n \sum_{j=1}^m \alpha_{ij} e_i \otimes e_j \in \bC^n \otimes \bC^m$ have norm $1$.  There is a state $s \in \cS(\cV_n \otimes_{\min} \cV_m)$ such that $s(u_{i1} \otimes v_{j1})=\alpha_{ij}$ for $1 \leq i \leq n$ and $1 \leq j \leq m$.  In fact, this state can be taken such that $s(u_{ij} \otimes v_{k\ell})=0$ whenever $j \neq 1$ or $\ell \neq 1$.
\end{mythe}

\begin{proof}
We may reduce to the case when $\alpha_{11} \geq 0$.  Indeed, we may choose $z \in \bT$ such that $z \alpha_{11} \geq 0$. Then we can first find $s' \in \cS(\cV_n \otimes_{\min} \cV_m)$ such that $s'(u_{i1} \otimes v_{j1})=z\alpha_{ij}$ for $1 \leq i \leq n$ and $1 \leq j \leq m$.  As the matrix $(\overline{z}u_{ij})$ is also unitary, the map $s:\cV_n \otimes_{\min} \cV_m \to \bC$ given by $s(u_{ij} \otimes v_{k\ell})=s'(\overline{z}u_{ij} \otimes v_{k\ell})=\overline{z}s'(u_{ij} \otimes v_{k\ell})$ also extends to a state on $\cV_n \otimes_{\min} \cV_m$; moreover, $s(u_{i1} \otimes v_{j1})=\alpha_{ij}$ and $s(u_{ij} \otimes v_{k\ell})=0$ whenever $j \neq 1$ or $\ell \neq 1$.  Hence, we may assume without loss of generality that $\alpha_{11} \geq 0$.

Let $r \in \bN$.  Define $h_0=e_1 \otimes e_1$ and $h_r=\sum_{i,j} \alpha_{ij} e_i \otimes e_j$.  Since $\la h_0,h_r \ra=\alpha_{11} \geq 0$, it follows that $\bR h_0+\bR h_r$ is a two-dimensional real Hilbert space, so there is a unitary $R:\bR h_0+\bR h_r \to \bR^2$ such that $R(h_0)=e_1$.  Since $\|h_0\|=\|h_r\|=1$, there is an orthogonal matrix $W \in M_2$ such that $We_1=Rh_r$.  It is clear that $W$ must be a rotation of the form $W=\begin{pmatrix} \cos \theta & -\sin \theta \\ \sin \theta & \cos \theta \end{pmatrix}$ for some $\theta \in [0,2\pi)$.  For $1 \leq j \leq r-1$, let $h_j=R^{-1}W_je_1$, where $W_j$ refers to the rotation $\begin{pmatrix} \cos \left( \frac{j\theta}{r} \right) & -\sin \left( \frac{j\theta}{r} \right) \\ \sin \left( \frac{j\theta}{r} \right) & \cos \left( \frac{j\theta}{r} \right) \end{pmatrix}$.  Then $W_pW_q=W_{p+q}$ and $W_p^T=W_{-p}$ for all $p,q \in \bZ$, so that, for $1 \leq j \leq r$,
$$\la h_j,h_{j-1} \ra=\left\la R^{-1}W_je_1,R^{-1}W_{j-1}e_1 \right\ra=\left\la W_1e_1,e_1 \right\ra=\cos \left(\frac{\theta}{r} \right),$$
Let $\psi=h_1 \otimes \cdots \otimes h_r \in (\bC^n \otimes \bC^m)^{\otimes r}$.  Define $U \in \cB((\bC^n)^{\otimes (r+1)})$ by cyclically shifting the tensors to the right by one position; i.e., for $x_0 \otimes \cdots \otimes x_r \in (\bC^n)^{\otimes (r+1)}$, we let $$U(x_0 \otimes \cdots \otimes x_r)=x_r \otimes x_0 \otimes x_1 \otimes \cdots \otimes x_{r-1}.$$
Then $U$ is unitary and can be identified as a unitary in $M_n(\cB((\bC^n)^{\otimes r}))$.  We define $V$ in the same way on $(\bC^m)^{\otimes (r+1)}$.  Then $U \otimes V$ is the unitary on $(\bC^n \otimes \bC^m)^{\otimes (r+1)}$ that permutes the copies of $\bC^n \otimes \bC^m$ by the cyclic right shift.  In particular, we have $$(U \otimes V)( (e_1 \otimes e_1) \otimes \psi)=h_r \otimes \psi_r,$$
where $\psi_r=h_0 \otimes \cdots \otimes h_{r-1}$.  In general, $$(U \otimes V)((e_i \otimes e_j) \otimes \psi)=h_r \otimes (e_i \otimes e_j) \otimes h_1 \otimes \cdots \otimes h_{r-1}.$$ There is a $*$-homomorphism $\pi:\cU_{nc}(n) \otimes_{\min} \cU_{nc}(m) \to \cB((\bC^n \otimes \bC^m)^{\otimes (r+1)})$ such that $$(\pi(u_{ij} \otimes v_{k\ell}))_{(i,j),(k,\ell)}=U \otimes V.$$
Define a state $s_r:\cV_n \otimes_{\min} \cV_m \to \bC$ by $$s_r(x)=\la \pi(x)\psi,\psi \ra, \, \forall x \in \cV_n \otimes_{\min} \cV_m.$$
Then $s_r(u_{i1} \otimes v_{j1})=\alpha_{ij} \la \psi,\psi_r \ra$ for $1 \leq i \leq n$ and $1 \leq j \leq m$.  We will show that $|\la \psi,\psi_r \ra|$ tends to $1$ as $r$ becomes large.

It is readily checked that $$\la \psi,\psi_r \ra=\la h_1,h_0 \ra \la h_2,h_1 \ra \cdots \la h_r,h_{r-1} \ra=\cos \left( \frac{\theta}{r} \right)^r.$$
In particular, $|\la \psi,\psi_r \ra|$ tends to $1$ as $r$ becomes large.  By dropping to a subsequence if necessary, we may assume that $(s_r)_{r=1}^{\infty}$ is a sequence of states converging pointwise.  Then $s:\cV_n \otimes_{\min} \cV_m \to \bC$ given by $s(x)=\lim_{r \to \infty} s_r(x)$ is a state on $\cV_n \otimes_{\min} \cV_m$ such that $s(u_{i1} \otimes v_{j1})=\alpha_{ij}$ for $1 \leq i \leq n$ and $1 \leq j \leq m$.

It remains to show that $s(u_{ij} \otimes v_{k\ell})=0$ whenever $j \neq 1$ or $\ell \neq 1$.  Consider the state $s_r$ above, corresponding to the unitaries $U \in \cB((\bC^n)^{\otimes (r+1)})$ and $V \in \cB((\bC^m)^{\otimes (r+1)})$ above.  Then $$(U \otimes V)((e_j \otimes e_{\ell}) \otimes \psi)=h_r \otimes (e_j \otimes e_{\ell}) \otimes h_1 \otimes \cdots \otimes h_{r-1}.$$
Note that $s_r(u_{ij} \otimes v_{k\ell})$ corresponds to the quantity $$\la (U \otimes V)((e_j \otimes e_{\ell}) \otimes \psi),(e_i \otimes e_k) \otimes \psi \ra.$$
Therefore, $$s_r(u_{ij} \otimes v_{k\ell})=\alpha_{ik} \la e_j \otimes e_{\ell},h_1 \ra \la h_1,h_2 \ra \cdots \la h_{r-1},h_r \ra.$$
Since $|\la h_1,h_2 \ra \cdots \la h_{r-1},h_r \ra| \leq 1$, we have $$|s_r(u_{ij} \otimes v_{k\ell})| \leq |\alpha_{ik}| |\la e_j \otimes e_{\ell},h_1 \ra| \leq |\la e_j \otimes e_{\ell},h_1 \ra|.$$
The angle between $h_1$ and $e_1 \otimes e_1$ is $\frac{\theta}{r}$, so it follows that $\|h_1-e_1 \otimes e_1\| \to 0$.  Thus, $|\la e_j \otimes e_{\ell},h_1 \ra| \to 0$ if $j \neq 1$ or $\ell \neq 1$.  This shows that $s_r(u_{ij} \otimes v_{k\ell}) \to 0$ if $j \neq 1$ or $\ell \neq 1$.  Hence, $s(u_{ij} \otimes v_{k\ell})=0$ when $j \neq 1$ or $\ell \neq 1$, which completes the proof.
\end{proof}

Using the embezzlement framework, we can distinguish the unitary correlation sets for $qs$ and $qc$  for all $n,m \geq 2$ and show that the unitary $qs$ sets are not closed.  The proof uses techniques found in \cite[Theorem 2.1]{CLP}.

\begin{cor}
For every $n,m \geq 2$, $B_{qs}(n,m) \neq B_{qa}(n,m)$.  In particular, $UC_{qs}(n,m) \neq UC_{qa}(n,m)$, and neither $UC_{qs}(n,m)$ nor $B_{qs}(n,m)$ are closed.
\end{cor}

\begin{proof}
Without loss of generality, we may assume that $n \leq m$.  Let $x=\frac{1}{\sqrt{n}} \sum_{i=1}^n e_i \otimes e_i \in \bC^n \otimes \bC^m$.  By Theorem \ref{embezzlement}, there is $X \in B_{qa}(n,m)$ with $X_{(i,1),(i,1)}=\frac{1}{\sqrt{n}}$ and $X_{(i,1),(j,1)}=0$ for $i \neq j$.  If $X \in B_{qs}(n,m)$, then there is a perfect embezzlement protocol in the tensor product model for $\frac{1}{\sqrt{n}} \sum_{i=1}^n e_i \otimes e_i$.  Let $U_{ij},V_{k\ell}$ and $\psi$ be as in the perfect embezzlement framework.  Then $$(U \otimes I_m)(I_n \otimes V) (e_1 \otimes \psi \otimes e_1)=\frac{1}{\sqrt{n}} \sum_{i=1}^n e_i \otimes \psi \otimes e_i.$$
Let $\alpha_1,\alpha_2,...$ be the Schmidt coefficients of $e_1 \otimes \psi \otimes e_1$ with respect to the decomposition $(\bC^n \otimes \cR_A) \otimes (\cR_B \otimes \bC^m)$, so that
\[ e_1 \otimes \psi \otimes e_1 = \sum_j \alpha_j x_j \otimes y_j,\]
where $\{ x_j \} \subseteq \bC^n \otimes \cR_A$ and $\{ y_j \} \subseteq \cR_B \otimes \bC^m$ are orthonormal sets.  Since
$$\frac{1}{\sqrt{n}} \sum_{i=1}^n e_i \otimes \psi \otimes e_i = (U \otimes I_m)(I_n \otimes V)(e_1 \otimes \psi \otimes e_1) = \sum_j \alpha_j (Ux_j) \otimes (Vy_j),$$  
the Schmidt coefficients of $e_1 \otimes \psi \otimes e_1$ must be the same as the Schmidt coefficients of $\frac{1}{\sqrt{n}} \sum_{i=1}^n e_i \otimes \psi \otimes e_i$.  But if $\alpha_0>0$ is the largest Schmidt coefficient of $e_1 \otimes \psi \otimes e_1$, then the largest Schmidt coefficient of $\frac{1}{\sqrt{n}} \sum_{i=1}^n e_i \otimes \psi \otimes e_i$ is at most $\frac{1}{\sqrt{n}} \alpha_0$, which is a contradiction.  Hence, $B_{qs}(n,m) \neq B_{qc}(n,m)$.

Finally, since any vector can be approximately embezzled, $X$ must be a limit of elements in $B_{qs}(n,m)$, so that $B_{qs}(n,m)$ is not closed.  It follows immediately that $UC_{qs}(n,m) \neq UC_{qa}(n,m)$ and that $UC_{qs}(n,m)$ is not closed.
\end{proof}

Suppose that $n \leq m$ and that $d_1,...,d_n>0$ are such that $\sum_{i=1}^n d_i^2=1$.  Consider any state $s:\cV_n \otimes_{\min} \cV_m \to \bC$ such that $s(u_{i1} \otimes v_{i1})=d_i$ and $s(u_{i1} \otimes v_{k1})=0$ for $i \neq k$.  Such a state arises from a perfect embezzlement protocol in the commuting model for the vector $\sum_{i=1}^n d_i e_i \otimes e_i$.  A surprising fact about the state $s$ is that its action on the elements $\{u_{ij} \otimes v_{k\ell}\}_{i,j,k,\ell}$ is necessarily unique.

\begin{pro}
\label{embezzledstatezeros}
Let $n \leq m$ and let $d_1,...,d_n>0$ be such that $\sum_{i=1}^n d_i^2=1$.  Suppose that $s:\cV_n \otimes_c \cV_m \to \bC$ is a state such that $s(u_{i1} \otimes v_{i1})=d_i$ for $1 \leq i \leq n$ and $s(u_{j1} \otimes v_{k1})=0$ for $j \neq k$.  Then $$s(u_{ij} \otimes v_{k\ell})=\begin{cases} d_i & i=k \leq n, \, j=\ell=1 \\ 0 & \text{otherwise}. \end{cases}$$
\end{pro}

\begin{proof}
Let $s$ be a state satisfying the equations given.  In the embezzlement setting, $s$ corresponds to the following: unitary operators $U:\bC^n \otimes \cR \to \bC^n \otimes \cR$ and $V:\cR \otimes \bC^m \to \cR \otimes \bC^m$ such that $(U \otimes I_m)(I_n \otimes V)=(I_n \otimes V)(U \otimes I_m)$, along with a unit vector $\psi \in \cR$ such that $s(u_{ij} \otimes v_{k\ell})=\la U_{ij}V_{k\ell}\psi,\psi \ra$ for all $i,j,k,\ell$.  We may write the product of $U \otimes I_m$ and $I_n \otimes V$ in block form as $$(U \otimes I_m)(I_n \otimes V)=(u_{ij}V)_{i,j=1}^n=(I_m \otimes V)(U \otimes I_n)=(v_{k\ell}U)_{k,\ell=1}^m.$$
With this identification in hand, one can check that $$\la U_{ij}V_{k\ell}\psi,\psi \ra=(\la (U \otimes I_m)(I_n \otimes V)(e_j \otimes \psi \otimes e_{\ell}),e_i \otimes \psi \otimes e_k \ra.$$
By Proposition \ref{embezzlecauchyschwarz}, we must have
$$U_{i1}V_{i1}\psi=d_i\psi, \, \forall 1 \leq i \leq n,$$
and similarly $$U_{i1}V_{k1}\psi=0, \, \forall i \neq k.$$
We aim to show that $\la U_{ij}V_{k\ell} \psi,\psi \ra=0$ whenever $(i,j,k,\ell) \neq (i,1,i,1)$.  We observe that
\begin{align*}
\sum_{j=1}^n \sum_{i=1}^n d_i e_i \otimes (U_{ij}^*\psi) \otimes e_j&=(U^* \otimes I_n) \left( \sum_{i=1}^n d_i e_i \otimes \psi \otimes e_i \right) \\
&=(I_n \otimes V) (e_1 \otimes \psi \otimes e_1)= \sum_{i=1}^n e_i \otimes (V_{i1}\psi) \otimes e_1. \\
\end{align*}
Comparing entries, we must have $U_{ij}^*\psi=0$ for all $j \neq 1$.  Similarly, if we instead apply $(I_n \otimes V^*)$, we obtain the following:
\begin{align*} \sum_{k,\ell=1}^m d_i e_k \otimes (V_{k\ell}^* \psi) \otimes e_{\ell}&=(I_n \otimes V^*) \left(  \sum_{\ell=1}^n d_{\ell} e_{\ell} \otimes \psi \otimes e_{\ell} \right) \\
&=(U \otimes I_n)(e_1 \otimes \psi \otimes e_1)=\sum_{k=1}^n e_1 \otimes U_{k1}\psi \otimes e_k. \\
\end{align*}
Comparing entries shows that $V_{k\ell}^*\psi=0$ if $\ell \neq 1$.  At this point, it follows that if $(i,j,k,\ell)$ is not equal to $(i,1,i,1)$, then $\la U_{ij}V_{k\ell}\psi,\psi \ra=0$, since $U_{ij}V_{k\ell}=V_{k\ell}U_{ij}$ and one of $U_{ij}^*\psi=0$ or $V_{k\ell}^*\psi=0$.  This completes the proof.
\end{proof}

This phenomenon applies to any maximally entangled unit vector in $\bC^n \otimes \bC^m$.  Recall that any simple tensor $x \otimes y \in \bC^n \otimes \bC^m$ has an associated map $T_{x,y}:\bC^m \to \bC^n$ given by $T_{x,y}(z)=\la z,y \ra x$.  Extending by bilinearity, for any $\alpha \in \bC^n \otimes \bC^m$, there is an associated linear map $T_{\alpha}:\bC^m \to \bC^n$; moreover, this is a 1-1 correspondence.  We will say that a unit vector $\alpha \in \bC^n \otimes \bC^m$ is \textbf{maximally entangled} if $\text{rank}(T_{\alpha})=\min\{n,m\}$.  Recall that any $x \in \bC^n \otimes \bC^m$ has a Schmidt decomposition $x=\sum_{i=1}^k d_i u_i \otimes v_i$, where $\{u_1,...,u_k\} \subseteq \bC^n$ is orthonormal and $\{v_1,...,v_k\} \subseteq \bC^m$ is orthonormal, and $d_i>0$ are in decreasing order; moreover, the $d_i$ are unique.  Then a unit vector is maximally entangled if and only if $k=\min\{n,m\}$.

\begin{cor}
\label{uniqueembezzlement}
Let $\alpha \in \bC^n \otimes \bC^m$, and let $X \in B_{qc}(n,m)$ be any matrix obtained by a perfect embezzlement protocol for $\alpha$ in the commuting model.  Then $X$ is unique if and only if $\alpha$ is maximally entangled in $\bC^n \otimes \bC^m$.
\end{cor}

\begin{proof}
First, suppose that $\alpha$ is maximally entangled.  Using the Schmidt decomposition, we write $\alpha=\sum_{i=1}^{\min(n,m)} d_i u_i \otimes v_i$, where $d_i>0$ for all $i$.  The proof of Proposition \ref{embezzledstatezeros} shows that the embezzlement correlation is unique when $u_i=e_i$ and $v_i=e_i$.  Thus, if $X \in B_{qc}(n,m)$ is a correlation matrix corresponding to a perfect embezzlement protocol for $\alpha$, then we can apply a unitary of the form $A \otimes B \in M_n \otimes M_m$ that sends $u_i \otimes v_i$ to $e_i \otimes e_i$, and we obtain a correlation matrix corresponding to a perfect embezzlement protocol for $\sum_{i=1}^n d_i e_i \otimes e_i$.  This correlation matrix is necessarily unique, so applying $A^* \otimes B^*$, the same result holds for $\alpha$.  Therefore, the matrix $X$ is unique.

Conversely, suppose that $\text{rank}(T_{\alpha})=p<\min \{n,m\}$.  We may write the Schmidt decomposition $\alpha=\sum_{i=1}^p d_i u_i \otimes v_i$.  Let $Y \in B_{qc}(p,p)$ be any matrix corresponding to a perfect embezzlement protocol for $\beta:=\sum_{i=1}^p d_i e_i \otimes e_i$, and let $U=(U_{ij})$ and $V=(V_{k\ell})$ be unitaries in $M_n(\bofh)$ and $M_m(\bofh)$ respectively such that $U_{ij}V_{k\ell}=V_{k\ell}U_{ij}$ for all $i,j,k,\ell$.  Now, the matrix $R=(U_{ij}) \oplus I_{n-p}$ is unitary in $M_n(\bofh)$.  Similarly, $S=(V_{k\ell}) \oplus I_{m-p}$ is unitary in $M_m(\bofh)$, and the entries of $R$ commute with the entries of $S$.  Therefore, there is a state $s:\cU_{nc}(n) \otimes_{\max} \cU_{nc}(m) \to \bC$ whose image in $B_{qc}(n,m)$ is of the form $Y \oplus I_{\min(n,m)-p}$, and this will give rise to a perfect embezzlement protocol for $\beta$.  Now, let $A \in M_n$ and $B \in M_m$ be unitary matrices such that $Ae_i=u_i$ and $Be_i=v_i$ for $1 \leq i \leq p$.  Using Proposition \ref{locallyunitarilyinvariant}, $\| \cdot \|_{qc}$ is locally unitarily invariant.  Thus, $X:=(A \otimes B)(Y \oplus I_{\min(n,m)-p}) \in B_{qc}(n,m)$, and this corresponds to a perfect embezzlement protocol for $\alpha$ in the commuting model.  If $Z$ is the matrix obtained in Theorem \ref{embezzlement} corresponding to $\alpha$, then $(A^* \otimes B^*)Z$ is the matrix obtained in Theorem \ref{embezzlement} corresponding to $\beta$.  Since only one column of $(A^* \otimes B^*)Z$ is non-zero, it is clear that $Y \oplus I_{\min(n,m)-p} \neq (A^* \otimes B^*)Z$.  Therefore, $X \neq Z$.  It follows that the correlation matrix for $\alpha$ is not unique.
\end{proof}

\begin{cor} 
\label{extreme}
Let $\alpha = \sum \alpha_{i,k} e_i \otimes e_k \in \bC^n \otimes \bC^m$ be a maximally entangled state and let $X= (x_{(i,j),(k,l)}) \in M_n \otimes M_m$ with
\[ x_{(i,j),(k,l)} = \begin{cases} \alpha_{i,k}, & \text{ when } j=l=1 \\ 0, & \text{ when } j \ne 1 \text{ or } l \ne 1. \end{cases}\] Then $X$ is an extreme point of $B_{qc}(n,m)$ and of $B_{qa}(n,m)$.
\end{cor}

\begin{proof}
Since $X \in B_{qa}(n,m) \subseteq B_{qc}(n,m)$, we need only show that $X$ is an extreme point of $B_{qc}(n,m)$.  Suppose that $X=\frac{1}{2}(Y+Z)$ where $Y,Z \in B_{qc}(n,m)$.  Let $\beta=\sum Y_{(i,1),(k,1)} e_i \otimes e_k$ and $\gamma=\sum Z_{(i,1),(k,1)} e_i \otimes e_k$.  Then $\beta$ and $\gamma$ are vectors in $\bC^n \otimes \bC^m$ with norm at most $1$.  Moreover, $\alpha=\frac{1}{2}(\beta+\gamma)$.  This forces $\beta=\gamma=\alpha$.  In particular, $Y$ and $Z$ correspond to a perfect embezzlement protocol in the commuting model for $\alpha$.  Since $\alpha$ is maximally entangled, Corollary \ref{uniqueembezzlement} shows that $Y=Z=X$.
\end{proof}

We now give a characterization for elements of $B_{loc}(n,m)$ corresponding to perfect embezzlement protocols in the commuting model.

\begin{mythe}
\label{localembezzlement}
Let $\alpha=\sum_{i,j} \alpha_{ij} e_i \otimes e_j \in \bC^n \otimes \bC^m$ be a unit vector, where $n,m \geq 2$.  The following are equivalent.
\begin{enumerate}
\item
There is a state $s \in \cS(\cV_n \otimes_{\min} \cV_m)$ such that $s(u_{i1} \otimes v_{j1})=\alpha_{ij}$ for all $i,j$ and $X:=(s(u_{ij} \otimes v_{k\ell})) \in B_{loc}(n,m)$.
\item
There exist unit $t_1,...,t_s \geq 0$ such that $\sum_{r=1}^s t_r=1$, and unit vectors $y_1,...,y_s \in \bC^n$ and $z_1,...,z_s \in \bC^m$ such that $$\alpha=\sum_{r=1}^s t_r y_r \otimes z_r.$$ 
\item
$\|\alpha\|_{\bC^n \otimes_{\pi} \bC^m}=1$.
\end{enumerate}
\end{mythe}

\begin{proof}
Since $\alpha$ is norm $1$ in the Hilbert space tensor product $\bC^n \otimes \bC^m$, we must have $\| \alpha \|_{\pi} \geq 1$.  Clearly by definition of the projective tensor product, (2) implies (3).  Suppose that (3) is true.  The open ball of radius $R>0$ about $0$ in $\bC^n \otimes_{\pi} \bC^m$ is the convex hull of the set $\{ x \otimes y \in \bC^n \otimes \bC^m: \|x\|\|y\| \leq R\}$.  For each $R>1$, we may write $\alpha=\sum_{r=1}^s t_r y_r \otimes z_r$ for some $t_1,...,t_s \geq 0$ with $\sum_{r=1}^s t_r=1$ and vectors $y_1,...,y_s \in \bC^n$ and $z_1,...,z_s \in \bC^m$ such that $\|y_r\|\|z_r\| \leq R$ for all $r$.  By a theorem of Caratheodory, we may always assume that $s \leq 2\dim(\bC^n \otimes \bC^m)+1=2nm+1$.  Since each $t_r \leq 1$ and $\|y_r\|,\|z_r\|<R$, by compactness and letting $R \to 1$, we may write $$\alpha=\sum_{r=1}^s t_r y_r \otimes z_r,$$
where $t_1,...,t_s \geq 0$ with $\sum_{r=1}^s t_r=1$ and $\|y_r\|=\|z_r\|=1$.  Therefore, $\| \alpha \|_{\pi} \leq 1$, so that $\| \alpha \|_{\pi}=1$.  It follows that (2) and (3) are equivalent.

Suppose that (1) holds, and let $X \in B_{loc}(n,m)$ be such that $X_{(i,1),(j,1)}=\alpha_{ij}$ for $1 \leq i \leq n$ and $1 \leq j \leq m$.  We assume without loss of generality that $n \leq m$.  Since the first column of $X$ is of norm $1$, we have $\|X\| \geq 1$; in particular, $\|X\|_{\pi} \geq 1$.  Therefore, $\|X\|_{\pi}=1$.  Let $P_n:M_n \to \bC^n$ and $P_m:M_m \to \bC^m$ be the linear maps defined by sending a matrix to its first column.  Then $P_n$ and $P_m$ are contractive.  Since the projective Banach space tensor norm is functorial, $P_n \otimes P_m:M_n \otimes_{\pi} M_m \to \bC^n \otimes_{\pi} \bC^m$ is contractive.  We observe that $(P_n \otimes P_m)(X)=\alpha$, so that $\|\alpha\|_{\bC^n \otimes_{\pi} \bC^m} \leq 1$.  The reverse inequality is immediate since $\|\alpha\|_{\bC^{nm}}=1$, which shows that (1) implies (3).

Suppose that (3) is true.  Let $X \in B_{qa}(n,m)$ be the matrix obtained in Theorem \ref{embezzlement} corresponding to a perfect embezzlement protocol in the commuting model for $\alpha$.  Then $X$ is a matrix with $0$'s in every column except that the first column has the entries of $\alpha$.  The inclusion maps $\iota_n:\bC^n \to M_n$ and $\iota_m:\bC^m \to M_m$ obtained by sending a vector $x$ to the matrix of $0$'s with first column $x$ are contractive, so $\iota_n \otimes \iota_m:\bC^n \otimes_{\pi} \bC^m \to M_n \otimes_{\pi} M_m$ is contractive.  Moreover, $(\iota_n \otimes \iota_m)(\alpha)=X$, which forces $X \in B_{loc}(n,m)$.  This shows that (3) implies (1).
\end{proof}

Since the state corresponding to perfect embezzlement in the commuting model for $\frac{1}{\sqrt{n}} \sum_{i=1}^n e_i \otimes e_i$ is unique and takes the form given in Theorem \ref{embezzlement}, we can separate $B_{loc}(n,m)$ and $B_q(n,m)$; moreover, we can also separate $B_{qc}(n,m)$ and $B_{qmax}(n,m)$.

\begin{cor}
For all $n,m \geq 2$, we have $UC_{loc}(n,m) \subsetneq UC_q(n,m)$.
\end{cor}

\begin{proof}
As usual, we may assume that $n \leq m$.  Let $X$ be the matrix obtained from the state $s \in \cS(\cV_n \otimes_{\min} \cV_m)$ in Proposition \ref{embezzledstatezeros}.  By Theorem \ref{localembezzlement}, $X \in B_{loc}(n,m)$ if and only if $\left\| \frac{1}{\sqrt{n}} \sum_{i=1}^n e_i \otimes e_i \right\|_{\bC^n \otimes_{\pi} \bC^m}=1$.  To see that this is not the case, Let $B:\bC^n \times \bC^m \to \bC$ be the bilinear form given by $$B(v,w)=\sum_{k=1}^n v_k w_k.$$
By Holder's inequality, $\|B\| \leq 1$ when regarded as a bilinear form from $\bC^n \times \bC^m$ into $\bC$.  It follows (see \cite[p. 23]{ryan}) that $\|X\|_{\pi} \geq \left| \sum_{i=1}^n \frac{1}{\sqrt{n}} B(e_i,e_i) \right|=\sqrt{n}$.  Hence, $X \not\in B_{loc}(n,m)$, which shows that $UC_{loc}(n,m) \neq UC_{qa}(n,m)$.  Since $X$ can be approximated by elements in $B_q(n,m)$ and $B_{loc}(n,m)$ is closed, we see that $B_{loc}(n,m) \neq B_q(n,m)$, so that $UC_{loc}(n,m) \neq UC_q(n,m)$.

Finally, we will show that $UC_{loc}(n,m) \subset UC_q(n,m)$.  We first note that $UC_{loc}(n,m) \subseteq \bC^{(2n^2+1)(2m^2+1)}$ is the closed convex hull of states arising from evaluation functionals on commutative $C^{\ast}$-algebras.  As in the proof of Theorem \ref{correlationnorms}, the resulting correlation in $UC_{loc}(n,m)$ will be of the form $$(\delta_z(X)\delta_z(Y))_{X \in \mathfrak{B}_n(U), \, Y \in \mathfrak{B}_m(V)},$$
where $U \in M_n(\bofh)$ and $V \in M_m(\bofh)$ are unitary and $K$ is a compact Hausdorff space such that $z \in K$ and $C^*(\mathfrak{B}_n(U) \cup \mathfrak{B}_m(V)) \simeq C(K)$.  We saw in the proof of Theorem \ref{correlationnorms} that $(\delta_z(X)\delta_z(Y))_{X,Y} \in UC_q(n,m)$.  By a theorem of Caratheodory, every element of $UC_{loc}(n,m)$ can be written as a finite convex combination of at most $2(2n^2+1)(2m^2+1)+1$ states of the form $(\delta_z(X)\delta_z(Y))_{X,Y}$.  Since $UC_q(n,m)$ is convex, it follows that $UC_{loc}(n,m) \subseteq UC_q(n,m)$, which completes the proof.
\end{proof}

\begin{cor}
For all $n,m \geq 2$, $B_{qc}(n,m) \neq B_{qmax}(n,m)$.  In particular, $\cV_n \otimes_c \cV_m \neq \cV_n \otimes_{\max} \cV_m$.  In fact, the identity map $\id:\cV_n \otimes_c \cV_m \to \cV_n \otimes_{\max} \cV_m$ fails to be $1$-positive.
\end{cor}

\begin{proof}
The extreme points of $B_{qmax}(n,m)$ are the extreme points of the unit ball of $M_n \otimes M_m$ in the operator norm, which is just the set of unitaries in $M_{nm}$.  By Corollary \ref{extreme}, there are proper contractions in $B_{qc}(n,m)$ that are extreme in $B_{qc}(n,m)$.  Therefore, $B_{qc}(n,m) \neq B_{qmax}(n,m)$.  This shows that $\id:\cV_n \otimes_c \cV_m \to \cV_n \otimes_{\max} \cV_m$ fails to be $1$-positive.
\end{proof}

\begin{cor}
None of the norms $\| \cdot \|_{loc}$, $\| \cdot \|_{qa}$ or $\| \cdot \|_{qc}$ are unitarily invariant.
\end{cor}

\begin{proof}
There is a unitary $W \in M_n \otimes M_m$ with $W \not\in B_{qc}(n,m)$; otherwise, we would have $B_{qmax}(n,m) \subseteq B_{qc}(n,m)$, since $B_{qmax}(n,m)$ is the closed convex hull of the unitaries in $M_n \otimes M_m$.  Note that $I_{nm}=I_n \otimes I_m \in B_{loc}(n,m)$ by Theorem \ref{crossnorm}.  However, $I_{nm}W=W \not\in B_{qc}(n,m)$, so that $\|I_{nm}W\|_t>1$ for $t \in \{loc,qa,qc\}$.  Hence, $\| \cdot \|_{loc}$, $\| \cdot \|_{qa}$ and $\| \cdot \|_{qc}$ are not unitarily invariant.
\end{proof}


\begin{thebibliography}{99}
\bib{blecherpaulsen}{article}{
author={D.P. Blecher},
author={V.I. Paulsen},
title={Tensor products of operator spaces},
journal={J. Funct. Anal.},
volume={99},
number={2},
date={1991},
pages={262--292},
}
\bib{boca}{article}{
author={F. Boca},
title={Free products of completely positive maps and spectral sets},
journal={J. Funct. Anal.},
volume={97},
number={2},
date={1991},
pages={251--263},
}
\bib{brown}{article}{
author={L. Brown},
title={Ext of certain free product $C^{\ast}$-algebras},
journal={J. Operator Theory},
volume={6},
pages={135--141},
year={1981},
}
\bib{choieffros}{article}{
	author={M.D. Choi},
	author={E.G. Effros},
	title={Injectivity and Operator Spaces},
	journal={J. Funct. Anal.},
	volume={24},
	number={2},
	date={1977},
	pages={156--209},
}
\bib{CLP}{article}{
author={R. Cleve},
author={L. Liu},
author={V.I. Paulsen},
title={Perfect embezzlement of entanglement},
journal={J. Math. Phys.},
volume={58},
number={1},
year={2017},
pages={0122014, 18pp.},
}
\bib{connes}{article}{
author={A. Connes},
title={Classification of injective factors.  Cases $II_1$, $II_{\infty}$, $III_{\lambda}$, $\lambda \neq 1$},
journal={Ann. Math. (2)},
volume={104},
number={1},
year={1976},
pages={73--115},
}
\bib{FKPTwep}{article}{
author={D. Farenick},
author={A.S. Kavruk},
author={V.I. Paulsen},
author={I.G. Todorov},
title={Characterisations of the weak expectation property},
journal={New York J. Math.},
year={to appear},
}
\bib{FP}{article}{
	author={D. Farenick},
	author={V.I. Paulsen},
	title={Operator system quotients of matrix algebras and their tensor products},
	journal={Math. Scand.},
	volume={111},
	number={2},
	pages={210--243},
	date={2012},
}
\bib{fritz}{article}{
author={T. Fritz},
title={Tsirelson's problem and Kirchberg's conjecture},
journal={Rev. Math. Phys.},
volume={24},
number={5},
year={2012},
pages={1250012, 67 pp.}
}
\bib{harris}{article}{
author={S.J. Harris},
title={A non-commutative unitary analogue of Kirchberg's conjecture},
journal={preprint (arXiv:1608.03229 [math.OA])},
year={2016},
}
\bib{junge}{article}{
author={M. Junge},
author={M. Navascues},
author={C. Palazuelos},
author={D. Perez-Garcia},
author={V.B. Scholz},
author={R.F. Werner},
title={Connes embedding problem and Tsirelson's problem},
journal={J. Math. Phys.},
year={2011},
volume={52},
number={1},
pages={012102, 12 pp.}
}
\bib{kavruk}{article}{
author={A.S. Kavruk},
title={Nuclearity related properties in operator systems},
journal={J. Operator Theory},
volume={71},
number={1},
year={2014},
pages={95--156},
}
\bib{KPTT}{article}{
	author={A.S. Kavruk},
	author={V.I. Paulsen},
	author={I.G. Todorov},
	author={M. Tomforde},
	title={Tensor products of operator systems},
	journal={J. Funct. Anal.},
	volume={261},
	number={2},
	pages={267--299},
	date={2011},
}
\bib{quotients}{article}{
	author={A.S. Kavruk},
	author={V.I. Paulsen},
	author={I.G. Todorov},
	author={M. Tomforde},
	title={Quotients, exactness, and nuclearity in the operator system category},
	journal={Adv. Math.},
	volume={235},
	pages={321--360},
	date={2013},
}
\bib{ozawa}{article}{
author={N. Ozawa},
title={About the Connes embedding conjecture: algebraic approaches},
journal={Jpn. J. Math.},
volume={8},
number={1},
date={2013},
pages={147--183},
}
\bib{paulsen02}{book}{
author={V.I. Paulsen},
title={Completely bounded maps and operator algebras},
series={Cambridge Studies in Advanced Mathematics},
publisher={Cambridge University Press, Cambridge},
year={2002},
pages={xii+300},
}
\bib{ryan}{book}{
author={R.A. Ryan},
title={Introduction to tensor products of Banach spaces},
series={Springer Monographs in Mathematics},
publisher={Springer-Verlag London Ltd., London},
year={2002},
pages={xiv+225},
}
\bib{slofstra}{article}{
author={W. Slofstra},
title={Tsirelson's problem and an embedding theorem for groups arising from non-local games},
journal={preprint (arXiv:1606.03140 [quant-ph])},
year={2016},
}
\bib{slofstra17}{article}{
author={W. Slofstra},
title={The set of quantum correlations is not closed},
journal={preprint (arXiv:1703.08618)},
year={2017},
}
\bib{tsirelson80}{article}{
author={B.S. Tsirelson},
title={Quantum generalizations of Bell's inequality},
journal={Lett. Math. Phys.},
volume={4},
number={2},
year={1980},
pages={93--100},
}
\bib{tsirelson93}{article}{
author={B.S. Tsirelson},
title={Some results and problems on quantum Bell-type inequalities},
journal={Hadronic J. Suppl.},
volume={8},
year={1993},
number={4},
pages={329--345},
}
\bib{vandam}{article}{
author={W. van Dam},
author={P. Hayden},
title={Universal entanglement transformations without communication},
journal={Phys. Rev. A},
volume={67},
number={6},
year={2003},
pages={060302, 3 pp.},
}
\end{thebibliography}
\end{document}